\documentclass[a4paper, twoside,keywordsasfootnote,addressatend,noinfoline]{article}

\usepackage{color}
\usepackage{srcltx}
\usepackage{a4wide}
\usepackage[english]{babel}
\usepackage{amssymb}
\usepackage{amsmath}
\usepackage{theorem}
\usepackage{epsfig}
\usepackage{imsart}

\theorembodyfont{\slshape}
\newtheorem{theor}{Theorem}

\newtheorem{propo}{Proposition}[section]
\newtheorem{lemma}[propo]{Lemma}
\newtheorem{defin}[propo]{Definition}

\newenvironment{proof}{\noindent{\scshape Proof.}}{\hspace{2mm} $\square$ \\}

\newcommand{\R}{\mathbb{R}}
\newcommand{\Z}{\mathbb{Z}}
\newcommand{\N}{\mathbb{N}}

\newcommand{\super}[1]{^{^{_{\,#1}}}}
\newcommand{\norm}[1]{|\!|#1|\!|}

\newcommand{\ind}{\mathbf{1}}

\newcommand{\Hh}{\hbox{H}}
\newcommand{\T}{\mathbb{T}}

\begin{document}

\begin{frontmatter}
\title {Ergodicity and hydrodynamic limits \\ for an epidemic model}
\runtitle  {Individual recovery process}
\author    {Lamia Belhadji}
\runauthor {Lamia Belhadji}
\address   {Laboratoire de Math\'ematiques Rapha\"el Salem, \\ UMR 6085, CNRS - Universit\'e de Rouen, \\
            Avenue de l'Universit\'e, BP. 12, \\ 76801 Saint Etienne du Rouvray, France.}

\begin{abstract} \ \   We consider two approaches to study 
the spread of infectious diseases within a spatially structured population distributed in 
social clusters.   According whether  we consider only the population of infected individuals or both populations of infected individuals and healthy ones, two  models are given to study an epidemic phenomenon.
Our first approach is at a microscopic level,  its goal is to determine if an epidemic may occur for those models.
The second one  is the derivation of hydrodynamics limits. By using the relative entropy method we prove that the empirical measures of infected and healthy individuals converge to a deterministic measure absolutely continuous with respect to the Lebesgue measure,   whose density is the solution of a system of reaction-diffusion equations.   \\ \\
\end{abstract}



\end{frontmatter}


\section{Introduction}
\label{Sec: introduction}

\indent We study an epidemic model describing the  course of a single disease within a spatially structured human population distributed in {\it social clusters} of {\it finite or infinite size}. 
That is, each site of the $d$-dimensional integer lattice $\Z^d$ is occupied by a cluster of {\it individuals}, each individual can be {\it healthy} or {\it infected } and the number of infected individuals at each cluster is either bounded or may be infinite. A cluster is said to be infected if it contains at least one infected individual and is said to be healthy otherwise. 
The first model we investigate is an extension of a process introduced in Schinazi (2002), and will be referred to as the \emph{cluster recovery process} (CRP). The second with another recovery mechanism extend a process introduced in Belhadji 
and Lanchier (2006), and will be referred to as the \emph{individual recovery process} (IRP).

\indent For both CRP (Schinazi, 2002) and IRP (Belhadji and Lanchier, 2006), the dynamics depends on three parameters, namely the outside infection rate $\lambda$ (the rate at which an individual infects healthy individuals of other clusters), the within infection rate $\phi$ (the rate at which an individual infects healthy individuals present in the same cluster), and the cluster size $\kappa$ (can be seen as the mean number of individuals having sustained contacts with a given individual).
 In both models, it is assumed that, once a cluster has at least one infected individual, infections within the cluster are a lot
 more likely than additional infections from the outside so we neglect the latter.
 The only difference between the CRP and the IRP is the recovery mechanism.
 For the CRP, all the infected individuals in a given cluster are simultaneously replaced by healthy individuals,
 which follows from the assumption that, once an infected individual is discovered, its social cluster rapidly recovers
 thanks to an antidote. For the IRP, we deal with the other extreme case, that is we assume that at most one infected individual recovers at once, that is the
 tracking system is not effective enough and the infection can spread within a given cluster before it is detected.
 In particular, the CRP and the IRP can be considered as spatial stochastic models for the transmission of
 infectious diseases in developed and developing countries, respectively.

\indent We assume that individuals within  the same cluster have repeated contacts whereas the individuals belonging to neighboring clusters have casual contacts only. This suggests that the infection spreads 
out faster within clusters than between them, this is the reason why we introduce in the IRP and CRP an other outside infection rate $\beta$ (at which an individual infects healthy individuals of other infected clusters). This allows us to take $\beta$ lower than $\phi$, to favorate  within infections. More general than the processes of Schinazi (2002) and Belhadji and Lanchier, we will assume that an outside infection may occur even if the cluster is already infected, 
and to avoid the condition that the number of infected individuals is bounded by $\kappa$ we will deal with IRP and CRP with infinite cluster size.

\indent The first aim of this paper  is to investigate the probability of an epidemic for both processes depending on the value of each of the three parameters $\lambda$, $\beta$ and $\phi$. 

\indent  In both the cluster recovery process (Schinazi, 2002) and individual recovery process (Belhadji and Lanchier, 2006),  
only the population of infected individuals is taken  into account; we will consider more general Markov processes evolving on the 1-dimensional lattice, without any restrictions on the clusters sizes, and with two types of particles, healthy and infected individuals. 

\indent In this model healthy  individuals get infected  with the same infection mechanism as in CRP, infected individuals recover at rate $1$ and moreover individuals are born, die and migrate, the migration of individuals (infected or healthy) speeded up by  renormalizing  parameter $N^2$. By using the relative entropy method, and in particular the works  Mourragui (1996),  Perrut (2000),  we will prove that the process admits hydrodynamic limits, that is by rescaling space  and time the densities of healthy and infected individuals evolve according to nonlinear reaction-diffusion equations.

\section{Presentation of the models and results}
\label{ICRP-Sec: resultats}

\indent In order to investigate the individual and cluster  recoveries processes with infinite cluster size, we start by introducing the evolution of the individual and cluster recoveries processes  with finite cluster size $\kappa\in \N$ denoted respectively  by IRP($\kappa$) and CRP($\kappa$).  The IRP$(\kappa)$ is a continuous-time Markov process in which the state at time $t$ is a function 
$\xi_t : \Z^d \longrightarrow \{0, 1, \ldots, \kappa\}\varsubsetneq \N$, with $\kappa$ denoting 
the common size of the clusters, and $\xi_t (x)$ indicates the number of infected individuals present in the cluster at time $t \geq 0$. To take into account the outside infections, we  introduce an interaction neighborhood. For any $x$, $z \in \Z^d$, $x \sim z$ indicates that site $z$ is one of the $2 d$ 
nearest neighbors of site $x$. Let  the  transition probability:  $\displaystyle p(x,y)= 1/2\,d\, \ind_{\{\norm{x-y}_1\, =\, 1\}},$ 
where $\norm{x - y}_1 = |x_1 - y_1| + \cdots + |x_d - y_d|$. Then, the state of site $x$ flips according to the  transition rates:
  \begin{eqnarray} \vspace{6pt}
 \label{ICRP-a}    0   \ &  \rightarrow  \ &  1    \hspace{5pt} \quad \qquad \textrm{at rate}  \qquad 2d\lambda \ \displaystyle \sum_{z\in \Z^d} \ p(x,z)\, \xi (z) \\ \vspace{12pt}
    \label{ICRP-b}    i   \  & \rightarrow \  &  i + 1  \qquad \textrm{at rate}  \qquad 2d\beta \ \displaystyle \sum_{z\in \Z^d} \ p(x,z)\, \xi (z) \ + \ i \,\phi \qquad i = 1,
2, \ldots, \kappa - 1 \\ \vspace{6pt}
   \label{ICRP-c}     i  \  & \rightarrow  \ &  i - 1   \qquad \textrm{at rate} \qquad i \qquad \hspace{4.1cm} i = 1, 2, \ldots, \kappa. \end{eqnarray} 
That is, a healthy cluster at site $x$ gets infected, i.e. the state of $x$ flips from 0 to 1, at rate $\lambda$ times the number of infected individuals present in the neighboring clusters.
In other respects, if there are $i$ infected individuals in the cluster $x$, $i = 1,2,\ldots,\kappa-1$,
then the state of $x$ flips from $i$ to $i+1$ at rate $\beta$ times the number of infected individuals present in the neighboring clusters plus $i\phi$ (each of infected individual infects healthy ones in the cluster $x$ at rate $\phi$). Finally, each infected individual recovers at rate 1 regardless of the number of infected individuals in its cluster.\\

\indent The CRP$(\kappa)$ is a Markov process $\eta_t : \Z^d \longrightarrow \{0, 1, \ldots, \kappa\}$, with $\eta_t (x)$ denoting the number of infected individuals at site $x$ at time $t \geq 0$, and whose evolution is obtained by replacing the transition \eqref{ICRP-c} above by
 \begin{equation}\label{ICRP-CRKrecovery}\begin{array}{lll}
i \ \rightarrow \ 0 & \qquad \textrm{at rate} & \qquad 1 \qquad \hspace{7pt} i = 1, 2, \ldots, \kappa. \end {array} 
\end{equation}
That is, all the infected individuals in a given cluster are now simultaneously replaced by 
healthy ones at rate 1, the infection mechanism modelled by \eqref{ICRP-a} and \eqref{ICRP-b} 
being unchanged.

\pagebreak

\noindent{\slshape{The graphical representation}} \\ \vspace{-7pt}

\indent An argument of Harris (1972) assures us of the existence and uniqueness of the models IRP($\kappa$) and CRP($\kappa$), for all $\kappa\geq 1$ finite.   For each $x$, $z \in \Z^d$ with $x \sim z$ and $i = 1, 2, \ldots, \kappa$, we let 
$\{T_n \super{x, z, i} : n \geq 1 \}$ (respectively, $\{{\tilde T}_n \super{x, z, i} : n \geq 1 \}$) denote the arrival times of independent Poisson processes with  rate $\lambda$ (respectively, $\beta$). 
To take into account the within infections, we introduce, for 
$x \in \Z^d$ and $i = 1, 2, \ldots, \kappa - 1$, a further collection of independent Poisson processes, denoted by $\{U_n \super{x, i} : n \geq 1 \}$, each of them has rate $\phi$.
Finally, for each $x \in \Z^d$ and $i = 1, 2, \ldots, \kappa$, we 
 let $\{V_n \super{x, i} : n \geq 1 \}$ be the arrival times of independent  rate 1 
Poisson processes.

\indent Given initial configurations $\xi_0$ and $\eta_0$, and the graphical representation introduced above, the process can be constructed as follows.
If there are at least $i$ infected individuals at site $x$ at time $T_n \super{x, z, i}$ 
(respectively, ${\tilde T}_n \super{x, z, i}$), then 
if site  $z$ is in state $j=0$  (respectively, $j, j=1,\ldots \kappa-1$) it flips to $j+1$ for both processes. In other respects, if there are $j$ infected individuals,  where $i \leq j \leq \kappa - 1$, 
at site $x$ at time $U_n \super{x, i}$, then one more individual gets infected in the cluster, i.e., the state of $x$ flips from $j$ to $j + 1$, for both processes.
 Finally, if there are $j$ infected individuals, $1 \leq j \leq \kappa$, at site $x$ at time 
$V_n \super{x, i}$, then the state of $x$ flips from $j$ to $j - 1$ if and only if $i \leq j$ for the process $\xi_t$,  while flips from $j$ to 0 if and only if $i = 1$ for the process $\eta_t$.
In particular, $\times_i$'s, $i = 2, 3, \ldots, \kappa$, have no effect on the process $\eta_t$.\\

\noindent{\slshape{The epidemic behavior of IRP($\infty$) and CRP($\infty$)}} \\ \vspace{-7pt}

\indent Assume now that each cluster may contain an infinite number of individuals. 
The resulting processes are denoted by IRP($\infty$) and CRP($\infty$).
The IRP($\infty$) (respectively, CRP($\infty$)) is a continuous-time Markov process 
in which the state at time $t$ is a function $\xi_t : \Z^d\rightarrow \N$, 
(respectively, $\eta_t : \Z^d\rightarrow \N$).  
The infection mechanism of the IRP$(\infty)$ and CRP($\infty$) 
is then described formally by setting $\kappa=\infty$ in the transitions \eqref{ICRP-a}, \eqref{ICRP-b}.
In the same way, the recovery mechanism of the IRP$(\infty)$ (respectively, CRP($\infty$)) is described by the transition \eqref{ICRP-c} (respectively, \eqref{ICRP-CRKrecovery}).\linebreak  To construct our processes we adopt an other point of view different from the graphical representation which moreover will allowed us to study their ergodicity.  We rely on techniques introduced in Chen (1992) to prove the existence and uniqueness of the IRP($\infty$) and the CRP($\infty$) when $\beta\leq \lambda$.

\indent We now discuss the effects of each of the three parameters, namely the outside infection \linebreak rates $\lambda$ and $\beta$, and the within infection rate $\phi$, on the probability of an epidemic for both models.
 
\indent From now on, we consider the processes starting with a single infected individual at site 0.
\begin{defin} We say that an epidemic may occur when
$$ P \,(\forall \,t \geq 0, \ \exists \ x \in \Z^d : \xi_t (x) \neq 0) \ > \ 0.$$
Otherwise, we say that there is no epidemic. \end{defin}
We prove by using basic coupling that the probability of an epidemic is nondecreasing with respect to the initial configuration and to each of the parameters $\lambda$, $\beta$, and $\phi$.

\indent Note that, when $\phi = 0$ and $\beta = 0$, there can be only one infected individual in each cluster so that both processes IRP($\infty$) and CRP($\infty$)  are identical and reduce to the basic contact process with infection rate $\lambda$,  in this case, there exists a critical value $\lambda_c \in (0, \infty)$ such that
 if $\lambda \leq \lambda_c$ then the processes converge in distribution to the ``all 0'' configuration;
 otherwise, an epidemic may occur. It follows by using basic coupling  that an epidemic may occur whenever $\lambda > \lambda_c$ regardless of the value of the parameters $\phi$ and $\beta$ and through a comparison with a branching random walk, we deduce that the processes IRP($\infty$) and CRP($\infty$) converge to the ``all 0'' configuration when 
$2 d \, \lambda < 1$.
When $\beta$ or $\phi$ are different from 0, the limiting behavior of the process is more complicated to predict due to the combined effects of the three birth rates. We can however by using the ergodicity criterion introduced in Chen (1992) we extend the result in the following way.
\begin{theor}\label{ICRP-tauxfaible}  If 
\begin{equation}\label{ICRP-TF}\phi \, + \, 2 \, d\, \lambda\, < \, 1,\end{equation}
then there is no epidemic for the IRP($\infty$) and CRP($\infty$) with parameters $(\lambda,\beta,\phi)$.
\end{theor}
\indent The cluster size being fixed, the analogue of Theorem \ref{ICRP-tauxfaible} for CRP($\kappa$) and IRP($\kappa$) is given by:

\begin{propo}
If \begin{equation}\label{ICRP-TF1}\phi \, + \, 2 \, d\, (\lambda\vee\beta)\, < \, 1,\end{equation}
then there is no epidemic for the IRP($\kappa$) and the CRP($\kappa$) with parameters $(\lambda,\beta,\phi)$ for all $\kappa \geq 1$. Note that this condition is uniform in the cluster size.
\end{propo}

\indent The ergodicity criterion established in Theorem \ref{ICRP-tauxfaible} shows that when \eqref{ICRP-TF} holds both IRP($\infty$) and CRP($\infty$) converge to the ``all 0'' configuration. 
Moreover, by using basic coupling, Theorem 1, Schinazi (2002) and Theorem 3, Belhadji and lanchier (2006) we prove that when $\phi$ is large enough an epidemic may occur for the IRP($\infty$) and CRP($\infty$):

\begin{theor}\label{IR-CR-phi}
For all $\kappa\geq 2$, $\phi \geq 0$ and $\beta\geq 0$, if  $\lambda > \lambda_c$ an epidemic may occur for the IRP($\infty$) and CRP($\infty$).
For all $\lambda$ and $\beta$ with $\beta \leq \lambda  <  1/ 2d $, 
there exists $\phi_c (\lambda,\beta) \in (0, \infty)$ such that: 
if $\phi < \phi_c (\lambda,\beta)$ there is no epidemic while an epidemic may occur for both processes
if $\phi > \phi_c (\lambda,\beta)$.
\end{theor}

\noindent As a consequence of ergodicity criterion \eqref{ICRP-TF1} and by analyzing the behavior of the processes IRP($\kappa$) and CRP($\kappa$) in the limiting case $\phi=\infty$ , the analogues of Theorem \ref{IR-CR-phi}  is given respectively by: 
\begin{propo} For all $\kappa\geq 2$, $\phi \geq 0$ and $\beta\geq 0$, if  $\lambda > \lambda_c$ an epidemic may occur for the IRP($\kappa$). 
For all $\kappa\geq 2$, $\lambda$ and $\beta$ with $\lambda\vee \beta  <  1/ 2d $, 
there exists $\phi_c (\lambda,\beta) \in (0, \infty)$ such that 
if $\phi < \phi_c (\lambda,\beta, \kappa)$ there is no epidemic,
while if $\phi > \phi_c (\lambda,\beta, \kappa)$ an epidemic may occur for IRP($\kappa$).\end{propo}
\begin{propo} For all $\kappa \geq 1$, $\phi \geq 0$ and $\beta\geq 0$, if $\kappa\, \lambda \leq \lambda_c$ there is no epidemic  while if  $\lambda > \lambda_c$ an epidemic may occur for the CRP($\kappa$). 
For all $\kappa \geq 2$, $\lambda > \kappa\lambda_c$ and $\lambda\vee \beta  <  1/ 2d $ there is $\phi_c (\lambda, \beta, \kappa) \in (0, \infty)$ such that if $\phi < \phi_c (\lambda,\beta, \kappa)$ there is no epidemic for the CRP($\kappa$) while if $\phi > \phi_c (\lambda, \beta, \kappa)$ an epidemic may occur.\end{propo}

\noindent{\slshape{Hydrodynamic limits for a two-species IRP with infinite cluster size}} \\ \vspace{-7pt}

\indent In the previous models,  only the population of infected individuals is taken  into account; we consider now a more general Markov process, without any restrictions on the clusters sizes, and with two types of particles, healthy and infected individuals. This epidemic model is a continuous-time Markov process $(\eta_t,\xi_t)_{t\geq 0}$ in which the state at time $t$ is a function $(\eta_t,\xi_t): \Z\rightarrow \N \times \N$,  where $\eta_t(x)$ and $\xi_t(x)$ are the respective numbers of healthy and infected individuals at site $x$ and at time $t$. The dynamics splits into two parts: diffusion and reaction. 
The diffusion represents the migration of individuals (infected or healthy) speeded up by a renormalizing  parameter $N^2$, 
it consists in independent symmetric random walks with nearest neighbor jumps, accelerated by  $N^2$. There is an interaction between healthy and infected individuals in
the reaction part, which describes births, deaths, recoveries and infections of individuals. 

\indent Our aim is to determine the limiting behavior of scaling processes as $N$ goes to infinity, in others words we will  prove  hydrodynamic limits for this epidemic model. The strategy consists first in restricting the study to the torus then by coupling method to extend the result to all space. 

To describe the evolution rules of the process,  we set   
$$ \eta^{x, +} (z) \ = \ \left\{\begin{array}{ll} \vspace{2pt}
 \eta (z) + 1 & \ \hbox{if} \ z = x, \\
 \eta (z)     & \ \hbox{if} \ z \neq x, \end{array} \right.
 \quad \hbox{and if}\ \eta(z)>0, \quad
 \eta^{x, -} (z) \ = \ \left\{\begin{array}{ll} \vspace{2pt}
 \eta (z) - 1 & \ \hbox{if} \ z = x, \\
 \eta (z) & \ \hbox{if} \ z \neq x,
  \end{array} \right. $$
and 
$$ \eta^{x, y} (z) \ = \ \left\{\begin{array}{ll} \vspace{2pt}
 \eta (x) - 1 & \ \hbox{if} \ z = x, \\ \vspace{2pt}
 \eta (y) + 1 & \ \hbox{if} \ z = y, \\
 \eta (z) & \ \hbox{otherwise.} \end{array} \right. $$
In other words,  $\eta^{x,+}$  (respectively, $\eta^{x,-}$) is the configuration obtained from $\eta$ by adding a particle at site $x$ (respectively, removing a particle at site $x$ if there is at least one). The configuration $\eta^{x, y}$ is obtained from $\eta$ by letting one particle 
jump from $x$ to $y$.  The formal infinitesimal generator is given for a cylinder function $f$  by
 \begin{equation}\label{hydro-process} 
\Omega f(\eta, \,\xi) \ =  \ \Omega\super {\cal R} f (\eta, \,\xi)\ +  \ N^2 \ \Omega\super {\cal  D} \, 
f (\eta, \,\xi)\end{equation}

where $\Omega\super {\cal D} \ = \  \Omega\super{{\cal D},1} \, + \,  \Omega\super{{\cal D},2}$, and $\Omega\super{{\cal D},1}$ (respectively, $\Omega\super{{\cal D},2}$ ) describes the migration of healthy (respectively, infected) individuals,
\begin{eqnarray}
\label{hydro-Hdeplacement}
\Omega\super {{\cal D},1}\, f (\eta, \,\xi) & = &  \displaystyle  \sum_{x, y\in \Z} \  
\displaystyle  \eta(x) \, p(x,y)\,  
\Big [f (\eta^{x, y}, \,\xi) - f (\eta, \,\xi) \Big]\\ \vspace{3pt}
\label{hydro-Ideplacement}\Omega\super {{\cal D},2}\,  f (\eta, \,\xi) & = &  \displaystyle   \sum_{x , y\in \Z} \  
\xi(x)  \, p(x,y)\,
\Big [f (\eta, \,\xi^{x, y}) - f (\eta, \,\xi)\Big],
\end{eqnarray}
$p(x,y)$ is a transition probability on the lattice $\Z$ such that a jump from site $x$ 
to site $y$ is allowed if and only if $x$ and $y$ are neighbors, given by 
$p(x,y)= {1\over 2}\ind_{\{|x-y|=1\}}$, and   

\begin{eqnarray*} \vspace{3pt}
\Omega\super {\cal R} f (\eta, \,\xi)
 & = & \displaystyle \sum_{x \in \Z} \ {\beta}_1 (\eta (x), \,\xi (x)) \Big [f (\eta^{x, +}, \,\xi)-
 f(\eta, \,\xi)\Big] \ + \ \delta_1 (\eta (x), \,\xi (x)) 
 \Big[f (\eta^{x, -}, \,\xi) - f(\eta, \,\xi)\Big] \\ \vspace{3pt} 
& + & \displaystyle \sum_{x \in \Z} \ {\beta}_2 (\eta(x), \,\xi(x)) \Big [f (\eta, \,\xi^{x, +}) - f (\eta, \,\xi)\Big ]  \ + \    \delta_2 (\eta (x), \,\xi (x)) 
\Big [f (\eta, \,\xi^{x, -}) - f (\eta, \,\xi)\Big]  \\ \vspace{3pt}
 & + & \displaystyle \sum_{x \in \Z} \ \xi (x) \Big [f (\eta^{x, +}, \,\xi^{x, -}) - 
f (\eta, \,\xi)\Big] \ + \ \ind_{\{\eta (x) > 0 \}} \ \phi \ \xi (x) \
  \Big [f (\eta^{x, -}, \,\xi^{x, +}) - f (\eta, \,\xi) \Big ]  \\ \vspace{3pt}
& + & \displaystyle \sum_{x \in \Z} \ \ind_{\{ \eta (x) > 0,\, \xi (x) = 0 \}} \Big (\lambda \ 
\sum_{y\sim x}  \xi (y)\Big) \ \Big [f (\eta^{x, -}, \,\xi^{x, +}) - f (\eta, \,\xi)\Big]\\ \vspace{3pt}
& + & \displaystyle \sum_{x \in \Z} \ \ind_{\{ \eta (x) > 0,\, \xi (x) > 0 \}} \Big (\beta \ 
\sum_{y\sim x}  \xi (y)\Big) \ \Big [f (\eta^{x, -}, \,\xi^{x, +}) - f (\eta, \,\xi)\Big], 
\end{eqnarray*}

where
\begin{equation}\label{Taux}\begin{array}{lll} \vspace{5pt}{\beta}_1(\eta(x), \xi(x))\ = \ \alpha_1\, (\eta(x)+\xi(x)),\  &\quad & \delta_1(\eta(x),\xi(x))= \kappa\, \eta(x)^2\, (\eta(x)+\xi(x)^2)
\\ \vspace{5pt}
{\beta}_2(\eta(x), \xi(x))\ = \ \alpha_2\, (\eta(x)+\xi(x)), \ &\quad& \delta_2(\eta(x),\xi(x))= \kappa\, \xi(x)^2\, (\eta(x)^2 +\xi(x)),\end{array}\end{equation}

\noindent and $\alpha_1,\alpha_2$ and $\kappa$ are positive coefficients. 
In other words, healthy (respectively, infected) individuals die at rate $\delta_1 (\eta, \xi)$
(respectively, $\delta_2 (\eta, \xi)$) and are born at rate ${\beta}_1(\eta, \xi)$ (respectively, $\beta_2(\eta, \xi)$);  a healthy cluster at site $x$ gets infected, that is the state of $x$ flips from 0 to 1, at rate $\lambda$ times the number of infected individuals present in the neighboring clusters. 
If there are $i \geq 1$ infected individuals in the cluster, 
then each of these individuals infects healthy individuals in the cluster at 
rate $\phi$; finally, each infected individual recovers at rate $1$ regardless of the number of infected individuals in its cluster.

\indent  Theorems 13.8 and 13.18 in Chen 1992, enable to establish  sufficient conditions for existence and uniqueness of the process $(\eta_t,\xi_t)_{t\in\R^+}$ whose evolution is described by the formal generator $\Omega$ in \eqref{hydro-process}.  We show that  conditions called  the first moment condition,  Lipschitz conditions, growing condition and moment condition are satisfied for the process. 

\indent  We first assume that healthy and infected individuals live on the space
$$\{x/N, \ x \in  \T_N\} $$
where $\T_N$ is the discrete torus  $\T_N = {\Z} / N {\Z}$ (i.e. sites $0$ and $N-1$ are neighbors). We make the distance between two neighboring sites converging to zero by letting $N$ goes to infinity. The evolution of the process is described by the generator 
\begin{equation}\label{hydro-OmegaN} \Omega_N \ =  \ \Omega\super {\cal R}_N \ +  \ N^2 \ 
\Omega\super {\cal  D}_N,\end{equation} 
where  $\Omega\super {\cal R}_N$ and  $\Omega\super {\cal  D}_N $ 
are the restrictions of $\Omega\super {\cal R}$ and  $\Omega\super {\cal  D}$  to $\T_N$. 
Let $\mu^N$ be the initial distribution of the process on $\N^{\T_N}\times \N^{\T_N}$ and  $S^N_t$ be the  semi-group associated to the generator $\Omega_N$. Using the relative entropy method (See Kipnis and Landim, 1999)  we will prove that the empirical measure $(\pi_t^N (\eta_t), \,\pi_t^N (\xi_t))$, defined by
\begin{equation}\label{mesures} 
\pi^N_t (\eta_t) \ = \ \displaystyle  {1 \over N} \ \sum_{x = 0}^{N - 1} \ \eta_t (x) \,\delta_{x / N}, \qquad
\pi^N_t (\xi_t) \ = \  \displaystyle {1 \over N} \ \sum_{x = 0}^{N - 1} \ \xi_t (x) \,\delta_{x / N}, 
\end{equation}
where $\delta_{x / N}$ is the Dirac measure at $x / N$,  converges in probability, on 
$D([0,T],M_+(\T)\times M_+(\T))$ (the space of right continuous functions with left limits taking values in $M_+(\T)\times M_+(\T)$ with $M_+(\T)$ is the space of finite positive measures on the torus $\T=[0,1)$ endowed with the weak topology), as $N$ goes to infinity, to a deterministic measure, absolutely continuous with respect to the Lebesgue measure, $(\rho_1 (t, u) \,du, \,\rho_2 (t, u) \,du)$,
with density  $(\rho_1 (\cdot, \cdot), \,\rho_2 (\cdot, \cdot))$  solution of the reaction-diffusion 
system \eqref{hydro-PDE}. The strategy consists in studying the entropy of the process with  respect 
to Poisson measures with parameter the expected ``good profile" 
$(\rho_1 (\cdot, \cdot), \,\rho_2 (\cdot, \cdot))$.

\indent For a density profile $\rho_1 (\cdot) \times \rho_2 (\cdot)$, on  $\T\times \T$,  we denote  by $\nu^N_{\rho_1 (\cdot)} \times \nu^N_{\rho_2 (\cdot)}$ the product of Poisson measures such that, for all $x \in \T_N$ and $k, j \in \N$,
$$ \begin{array}{rcl} \vspace{10pt}
\Big (\nu_{\rho_1 (\cdot)}^N \,\times \,\nu_{\rho_2 (\cdot)}^N\Big)\, \left \{ (\eta, \xi), \,\eta(x) = k, \,\xi (x) = j \right  \} & = &
\displaystyle \frac{(\rho_1 (x / N))^k}{k!} \ \exp \,(- \rho_1 (x / N)) \\
&  & \times \ \displaystyle \frac{(\rho_2 (x / N))^j}{j!} \ \exp \,(- \rho_2 (x / N)).
\end{array} $$
The family of  measures $(\nu_{\rho}^N\times\nu_{\rho}^N)$ with constant parameter $\rho>0$ is 
invariant for the independent 
random walks which govern the migration of individuals, this is 
 why we study the entropy variation with respect to these reference measures.

\indent We define the entropy of  $\mu^N$ on $\N^{\T_N}\times \N^{\T_N}$ with respect 
to $(\rho^1(.),\rho^2(.))$ by 
\begin{equation}\label{entro1}
\Hh\, \left[ \mu^N|  \nu^N_{\rho_1(.)}\times \nu^N_{\rho_2(.)}\right ] = \int \log \left ( {d\,\mu^N \over d \, (\nu^N_{\rho_1(.)}\times \nu^N_{\rho_2(.) })}\right)\, d\, \mu^N(\eta,\xi).
\end{equation}

\indent For a cylinder function $h$ on $\N^{\T_N}\times \N^{\T_N}$
\begin{equation}\label{hydro-tilde} 
{\widetilde h} (a,b) \ = \ \int h (\eta, \,\xi) \ d(\nu_a^N \times\nu_b^N) (\eta,\xi).
\end{equation}

\begin{theor}\label{hydro-hth}
Assume that there exists smooth positive functions $m_1 (\cdot)$ and $m_2 (\cdot)$, defined on the 
torus $\T$, such that
\begin{equation}
\label{ci}
\limsup_{N \to \infty} \  \displaystyle {1 \over N} \, \Hh \, \left[\,\mu^N \,| \,\nu_{m_1 (\cdot)}^N \,\times 
\,\nu_{m_2 (\cdot)}^N \right] \ = \ 0.
\end{equation}
 Then for all functions $G_1 (\cdot)$ and $G_2 (\cdot)$ continuous on $\T$, $\delta > 0$ and 
 $t \in [0, T]$, we have
 $$ \displaystyle \lim_{N \to \infty} \ \mu^N S^N_t \
 \left\{(\eta, \,\xi) : \left | \displaystyle  {1 \over N} \ \displaystyle \sum_{x=0}^{N-1} \ \eta(x) \,G_1 (x / N) -
 \int_{0}^1 G_1 (\theta) \,\lambda_1 (t, \,\theta) \,d\theta \,\right| > \,\delta \right. $$
 $$\hbox{and} \left. \quad \left|\,\displaystyle \displaystyle  {1 \over N}\, \sum_{x=0}^{N-1} \ \xi(x) \,G_2 (x / N) -
 \int_0^1 G_2 (\theta) \,\lambda_2 (t, \theta) \,d\theta \,\right| > \,\delta \,\right\} \ = \ 0 $$
where $(\lambda_1 (t, \cdot), \lambda_2 (t, \cdot))$ is the unique smooth solution of the system:
 \begin{equation}
 \label{hydro-PDE} \quad \partial_t
 \left(\begin{matrix}\vspace{8pt}
 \lambda_1  \\  \vspace{8pt} \lambda_2
 \end{matrix} \right)
 \ = \  \displaystyle {1 \over 2} \ \Delta
 \left(\begin{matrix} \vspace{8pt}
 \lambda_1  \\ \vspace{8pt}
 \lambda_2
 \end{matrix} \right) \ + \
 \left(\begin{matrix}\vspace{8pt}
 {\widetilde {\beta}}_1 (\lambda_1, \,\lambda_2) - {\widetilde \delta}_1 (\lambda_1, \,\lambda_2)
 +{\widetilde g} (\lambda_1, \,\lambda_2) 
  \\ \vspace{8pt}
  {\widetilde {\beta}}_2 (\lambda_1, \,\lambda_2) - {\widetilde \delta}_2 (\lambda_1, \,\lambda_2)-
  {\widetilde g} (\lambda_1, \,\lambda_2) 
 \end{matrix} \right), 
\end{equation}
with initial conditions $\lambda_1 (0, \,d\theta) = m_1 (\theta)$, and $\lambda_2(0, \,d\theta) = m_2 (\theta)$; and $g$ is a function on $\N\times \N$ defined by 
\begin{equation}\label{hydro-function}
g(\eta(z),\xi(z))= \Big (1- \phi\, \ind_{\{\eta(z) > 0\}}\Big )\, 
\xi(z) - \ind_{\{\eta(z)>0\}}\Big (\lambda \, \ind_{\{\xi(z)=0\}} + \beta \, \ind_{\{\xi(z)>0\}}\Big) \displaystyle \sum_{y\sim z} \xi(y).
\end{equation}
\end{theor}

\noindent{\slshape{Extension to infinite volume}} \\ \vspace{-7pt}

\indent By a coupling method, we will extend Theorem \ref{hydro-hth} to infinite volume. We will prove that two processes, one defined on $\Z$ and the other one on $\T_{C\, N}=\{-C\, N, \ldots, C\, N\}$, are ``close" when $C$ is large. Following Landim and Yau (1995) we define the specific entropy of a measure $\mu$ with respect to a measure $\nu$ on $\Z$   
\begin{equation}\label{hydro-newentropy} 
{\cal H}_N \left[\,\mu \,| \,\nu \right]\ = \ {1\over N} \sum_{n\geq 1} \Hh\left [\mu^n| \nu^n \right ]\, e^{-\theta\,n /N},
\end{equation}
where $\theta>0$ is fixed  and $\mu^n$ and $\nu^n$ are the respective restrictions of $\mu$ and $\nu$ to $\Lambda_n=\{-n,\ldots, n\}$.  Let $\tilde {S}^N_t$ be the semi-group associated to the generator 
$\Omega$ of the process $(\eta_t,\xi_t)_{t\geq 0}$ given in \eqref{hydro-process}.
\begin{theor}\label{hydro-hthiv}
We consider a sequence of initial distributions $(\mu^N)_{N\in\N}$ on $\Z^{\N}\times \Z^{\N}$ such that there exists $M>0$ with $\mu^N(\eta(x)+\xi(x))\leq M$ for all $x\in \Z$, and  smooth positive functions 
$m_1 (\cdot)$ and $m_2 (\cdot)$, defined on $\R$, satisfying
\begin{equation}
\label{ciiv}
\limsup_{N \to \infty} \  \displaystyle {1 \over N} \, {\cal H}_N \, \left[\,\mu^N \,| \,\nu_{m_1 (\cdot)}^N \,\times 
\,\nu_{m_2 (\cdot)}^N \right] \ = \ 0.
\end{equation}
 Then for all functions $G_1 (\cdot)$ and $G_2 (\cdot)$ continuous on $\R$, $\delta > 0$ and 
 $t \in [0, T]$, we have
 $$ \displaystyle \lim_{N \to \infty} \ \mu^N {\tilde S^N_t} \
 \left\{(\eta, \,\xi) : \left | \displaystyle  {1 \over N} \ \displaystyle \sum_{x \in\Z} \ \eta(x) \,G_1 (x / N) -
 \int_{\R} G_1 (\theta) \,\lambda_1 (t, \,\theta) \,d\theta \,\right| > \,\delta \right. $$
 $$\hbox{and} \left. \quad \left|\,\displaystyle \displaystyle  {1 \over N}\, \sum_{x \in\Z} \ \xi(x) \,G_2 (x / N) -
 \int_{\R} G_2 (\theta) \,\lambda_2 (t, \theta) \,d\theta \,\right| > \,\delta \,\right\} \ = \ 0 $$
where $(\lambda_1 (t, \cdot), \lambda_2 (t, \cdot))$ is the unique smooth solution of the 
system \eqref{hydro-PDE},  with initial conditions $\lambda_1 (0, \,d\theta) = m_1 (\theta)$, and $\lambda_2(0, \,d\theta) = m_2 (\theta)$.
\end{theor}
\section{Proof of Theorems \ref{ICRP-tauxfaible} and \ref{IR-CR-phi}}
\label{ICRP-Sec: proof}

\noindent{\slshape{Proof of Theorem \ref{ICRP-tauxfaible}.}}  The aim of this section is to prove that, when  
$\phi  +  2 \, d\, \lambda <  1$  the processes converge to the all ``0" configuration, this result will be deduce from an ergodicity criterion established in Chen (1992).

\indent For any integer $n \geq 1$, we set $\Lambda_n = \{- n, \ldots, n \}^d$, $\tilde{\lambda}= 2d\lambda$,  \ $\tilde{\beta}= 2d\beta$. We consider the sequence of processes $(\xi^n_t)_{n\geq 0}$ (respectively, 
$(\eta^n_t)_{n\geq 0}$) defined on $\N^{\Lambda_n}$ as the restriction of $\xi_t$ (respectively, 
$\eta_t$) to $\Lambda_n$ with 
generator ${\hat \Omega}_n$ (respectively, $\bar {\Omega}_n)$. For any  cylinder function 
$f$ of the configuration $\xi$,
$${\hat \Omega}_n f(\xi) \ = \ {\hat \Omega}_n\super{1} f (\xi) \ + \  {\Omega}_n\super 2 f (\xi),$$
$$ \begin{array}{rcl} \vspace{3pt}
\hat \Omega_n\super 1 f (\xi)
 & = & \displaystyle \phi \sum_{x\, \in\, \Lambda_n} \xi(x)  
 \Big [f (\xi^{x, +}) - f(\xi)\Big] 
  +  \displaystyle \sum_{x\, \in\,\Lambda_n} \xi(x) \Big[f (\xi^{x,-}) -
 f (\xi)\Big] 
 \end{array} $$
  and 
\begin{eqnarray*}\vspace{4pt}
 \Omega_n\super 2 f (\xi) & = & \displaystyle  \sum_{x\in\Lambda_n} 
   \left (\sum_{ \substack {y\sim x \\ y \, \in \, \Lambda_n}} \xi(y)\right)  \Big(  \lambda\, \ind_{\{\xi(x)\, = \, 0\}} 
 \, + \,  \beta  \, \ind_{\{ \xi(x)\, > \, 0\}}\Big )
 \Big [f (\xi^{x, +}) - f (\xi) \Big]\\ \vspace{4pt}
&  =&  \displaystyle \sum_{x, y\, \in\,  \Lambda_n} \  \displaystyle  
p(y,x)\,\xi(y) \, \Big (\tilde {\lambda}\, \ind_{\{\xi(x)\,=\, 0\}} + \tilde {\beta}\, 
\ind_{\{ \xi(x)\,>\, 0\}}\Big ) 
 \Big [f (\xi^{x, +}) - f (\xi) \Big].\\ 
\end{eqnarray*}
For any cylinder function $f$ of the configuration $\eta$,
$${\bar \Omega}_n f(\eta)\ = \ {\bar \Omega}_n\super{1} f (\eta) \ + \ {\Omega}_n\super 2 f (\eta),$$
where
$$\begin{array}{rcl}\vspace{3pt}
{\bar \Omega}_n\super 1 f (\eta)
& = & \displaystyle \phi \sum_{x\, \in\, \Lambda_n} \eta(x)  
\Big [f (\eta^{x, +}) - f(\eta)\Big] 
+ \displaystyle \sum_{x\,\in\,\Lambda_n} \Big[f(\eta^{x})-f(\eta)\Big].
\end{array}$$
Given a constant $M > 1$ which can be as close to $1$ as desired, we set
\begin{equation}\label{ICRP-sequence}
 k_x \ = \ \sum_{n = 0}^{\infty} \ M^{- n} \ p^{(n)} (x, 0) \qquad \hbox{for all}  \ x \in \Z.
\end{equation}
Since $p (x, y)$ is translation invariant with $p (x, x) = 0$,  we have 
\begin{equation}\label{ICRP-controlled}
\sum_{y \in \Z} \ p (x, y) \ k_y \ \leq \ M \,k_x \qquad \hbox{and} \qquad \sum_{x \in \Z} \ 
k_x \ < \ + \infty,
\end{equation}
Now, given a site $x \in \Z$ and two configurations $\xi_1$ and $\xi_2$, we set
$$ \rho_x(\xi_1(x))\ = \ \xi_1(x), \qquad q_x (\xi_1) \ = \ \xi_1 (x) \,k_x \qquad \hbox{and} \qquad q_x (\xi_1, \xi_2) \ = \ |\xi_1 (x) - \xi_2 (x)| \,k_x. $$
We construct our processes on   
$$E_{0}=\{\xi \in  E= \N^{\Z}: \, q(\xi) \, = \,  \sum_{x\in\Z} \rho_x(\xi(x)) k_x < \infty \}.$$
The following theorem is a an adaptation of Theorems 13.8, 13.18 and 14.3. 

\begin{theor}\label{ICRP-Chen2} 
For every $1\leq n\leq m$, there exists a coupling generator 
$\hat {\Omega}_{n,m}$ of $\hat{\Omega}_n$ and $\hat {\Omega}_m$ such that for any 
$\,  z\in\Lambda_n$ and any $\xi_1,\xi_2\in E_{0},$
\begin{equation}\label{ICRP-cond3}
\hat{\Omega}_{n,m}\, q_z(\xi_1,\xi_2)\ \leq \  \sum_{x\in \Lambda_n} \ c_{xz}\,  q_x(\xi_1,\xi_2)\ + \sum_{x\in \Lambda_m \setminus \Lambda_n}\ 
g_{xz} \, q_x (\xi_2), 
\end{equation}
where the non-diagonal elements of matrices  $(c_{xy})_{x,y\in\Lambda_n}, \,  (g_{xy})_{x,y\in\Lambda_m},$ are all non-negative, and    
\begin{equation}\label{ICRP-cond44}\displaystyle \lim_{n\rightarrow \infty}\ \displaystyle \sup_{x\in\Lambda_n}\ \sum_{y\in\Lambda_n} \ (c_{xy}+ g_{xy}) <\  +\infty.\end{equation}
Assume additionally that the coefficients $(c_{xy})$ given in \eqref{ICRP-cond3} also satisfy
\begin{equation}\label{ICRP-cond4}
\exists \ \alpha > 0, \quad  \lim_{n\rightarrow \infty} \ \sup_{x\in\Lambda_n}  \sum_{y\in\Lambda_n} \ c_{xy}\  < \ -\alpha\  <\  0,
\end{equation}
\begin{equation}\label{ICRP-cond5}
\exists \ K \ < \ \infty, \quad   \lim_{n\rightarrow \infty} \ \sup_{x\in\Lambda_n} \sum_{y\in\Lambda_n}\  |c_{xy}| \  <  \ K.
\end{equation}
Then the Markov process $(\xi_t)_{t\geq 0}$  has at most one stationary distribution  $\pi$ on $(E, {\cal E})$ satisfying 
\begin{equation}\label{ICRP-norme}
 \pi\, q\ = \ \int_{E_0} \pi \, (d\zeta)\, q(\zeta)\ < \  \infty.
\end{equation}
 \end{theor}

\indent Our main tool is to use repeatedly basic coupling of the different generators describing all the aspects of the model under study. We write in detail the first one, the others are built in the same spirit.    

\indent To check Condition \eqref{ICRP-cond3}, we use  basic coupling.
Let $n$ and $m$ be two integers such that $1\leq n\leq m$.
We denote by $\hat \Omega_{n,m}\super 1$ the coupled generator associated to  
$\hat \Omega_n\super 1$ and $\hat\Omega_m\super 1$, and by $ \Omega_{n,m}\super 2$ the coupled 
generator associated to $\Omega_n\super 2$ and $\Omega_m\super 2$.   
In the same way, we define the coupled generator ${\bar \Omega}_{n,m}\super{1}$.
To lighten our calculations, we set 
$$a_i(x,y)= \xi_i(y) \Big( \lambda\  \ind_{\{ \xi_i(x)=0\}} \ + \ \beta\ \ind_{\{ \xi_i(x)> 0\}} \Big ) \quad \hbox{for}\ i=1,2 \ \hbox{and }\ x,y\in\Z.$$
We define the coupling $\Omega_{n,m}\super 2$ describing the infections originated from neighboring sites as follows
$$\begin{array}{rcl} \vspace{3pt}
  \Omega_{n,m}\super 2 f (\xi_1,\xi_2)
& = & 2 d\,  \displaystyle  \sum_{x, y\, \in\, \Lambda_n}   p(x,y)\, (a_1(x,y)\wedge a_2(x,y))\,   \Big [f (\xi^{y, +}_1,\xi^{y, +}_2 ) - f(\xi_1,\xi_2)\Big] \\ \vspace{3pt}
& + &  2 d\,  \displaystyle  \sum_{x,y\, \in\, \Lambda_n}  p(x,y)\, (a_1(x,y)-a_2(x,y))^+\,  
 \Big [f (\xi^{y, +}_1,\xi_2 )-f(\xi_1,\xi_2)\Big]\\ \vspace{3pt}
 & + &  2 d\,  \displaystyle  \sum_{x,y\, \in\, \Lambda_n} p(x,y) \,    
(a_2(x,y)- a_1(x,y))^+\,  
 \Big [f (\xi_1,\xi^{y, +}_2 ) - f(\xi_1,\xi_2)\Big]\\ \vspace{3pt} 
& + & 2 d\,  \displaystyle  \sum_{x\, \in\, \Lambda_m \setminus \Lambda_n} \sum_{y\, \in\,  \Lambda_m} p(x,y) \,   a_2(x,y)    
 \Big [f (\xi_1,\xi^{y, +}_2 ) - f(\xi_1,\xi_2)\Big] \\ \vspace{3pt}
 & + & 2 d\,  \displaystyle  \sum_{x\, \in\, \Lambda_n} \sum_{y\in \Lambda_m\setminus \Lambda_n}
 p(x,y)\,  a_2(x,y)
 \Big [f (\xi_1,\xi_2^{y,+} ) - f(\xi_1,\xi_2)\Big]
 \end{array} $$
The coupled generator $\Omega_{n,m}\super 2$ describes the outside infections from site $x$ to $y$ 
where $x\sim y$ and ($x,y\in \Lambda_n $) or ($x\in \Lambda_m\setminus \Lambda_n$ and $y\in\Lambda_m$) 
or ($x\in \Lambda_n$ and $y\in \Lambda_m\setminus \Lambda_n$), for the processes whose  generators 
are $\Omega\super 2_n$ and $\Omega\super 2_m$. 

\indent We now deal with the coupled  generators $\hat\Omega_{n,m}\super 1$ (respectively, $\bar{\Omega}_{m,n}\super 1$) defined as the sum  of 
$ \hat \Omega_{n,m}\super {1,i}$,  $i=1,2$  (respectively, $ \bar{ \Omega}_{n,m}\super {1,i}$, $i=1,2$) with the coupled generator $ \hat{ \Omega}_{n,m}\super {1,1}$ (respectively, $\hat{ \Omega}_{n,m}\super {1,2}$) describing the within infections (respectively, recoveries) at site $x\in\Lambda_n$.
However, the following coupled generators 
$$ {\bar \Omega}_{n,m}\super {1,1} f (\eta_1,\eta_2) = {\hat \Omega}_{n,m}\super {1,1} f (\eta_1,\eta_2) \qquad \hbox{and} \qquad {\bar \Omega}_{n,m}\super {1,2} f (\eta_1,\eta_2) = \displaystyle  \sum_{x\, \in\, \Lambda_m}   \Big [f (\eta^x_1,\eta^x_2) - f(\eta_1,\eta_2)\Big],$$
describe the within infections and cluster recovery in a given site for the CRP($\infty$). For sites $x, y, z \in \Z$, we set
 $$ \begin{array}{rcl}
b_z (x, y) & = & (a_1 (x, y) - a_2 (x, y))^+ \ [q_z (\xi_1 \super{x, +}, \xi_2) - q_z (\xi_1, \xi_2)] \vspace{5pt} \\ & + &
    (a_2 (x, y) - a_1 (x, y))^+ \ [q_z (\xi_1, \xi_2 \super{x, +}) - q_z (\xi_1, \xi_2)]. \end{array} $$
First of all, we observe that
\begin{equation}
\label{ICRP-q1}
   q_z (\xi_1 \super{x, +}, \xi_2) \ = \ q_z (\xi_1, \xi_2 \super{x, -}) \ = \
  \left\{\begin{array}{ll}
   q_z (\xi_1, \xi_2) + k_z & \hbox{when $\xi_1 (x) \geq \xi_2 (x)$ and $x = z$} \vspace{5pt} \\
   q_z (\xi_1, \xi_2) - k_z & \hbox{when $\xi_1 (x) < \xi_2 (x)$ and $x = z$}\vspace{5pt} \\
   0  & \hbox{when} \ x \neq z,
 \end{array} \right.
\end{equation}
 while
\begin{equation}
\label{ICRP-q2}
   q_z (\xi_1, \xi_2 \super{x, +}) \ = \ q_z (\xi_1 \super{x, -}, \xi_2) \ = \
  \left\{\begin{array}{ll}
   q_z (\xi_1, \xi_2) + k_z & \hbox{when $\xi_1 (x) \leq \xi_2 (x)$ and $x = z$} \vspace{5pt} \\
   q_z (\xi_1, \xi_2) - k_z & \hbox{when $\xi_1 (x) > \xi_2 (x)$ and $x = z$} \vspace{5pt} \\
   0 & \hbox{when} \  x \neq z.
\end{array} \right.
\end{equation}
In particular,  by decomposing according to whether $\xi_1 (x)$ and $\xi_2 (x)$ are different from or equal to $0$, we obtain
 $$\begin{array}{rcl}\vspace{5pt}
b_z (x, y) & = & \lambda\ |\xi_1 (y) - \xi_2 (y)| \ k_z \ \ind_{\{\xi_1 (x)  =  \xi_1 (x)  =  0 \}} 
\ + \ \beta\ |\xi_1 (y) - \xi_2 (y)| \ k_z \ \ind_{\{\xi_1 (x)>0, \  \xi_1 (x)>0\}} \vspace{5pt} \\ 
& + & (\beta\,  \xi_2 (y) - \lambda\, \xi_2 (y)) \ k_z \ \ind_{\{\xi_2 (x)  >  \xi_1 (x)  =  0\}}  \ + \ (\beta\, \xi_1 (y) - \lambda \, \xi_1 (y)) \ k_z \ 
\ind_{\{\xi_1 (x)  >  \xi_2 (x)  =  0 \}}\vspace{8pt} \\ 
 & \leq &  \lambda \ |\xi_1 (y) - \xi_2 (y)| \ k_z \ = \ q_y (\xi_1, \xi_2) \ k_z / k_y \end{array} $$
when $x = z$, and $b_z (x, y)=0$  when $x \neq z$. We conclude that
\begin{equation}
\label{ICRP-bound}
b_z (x, y) \ \leq \lambda\, \ q_y (\xi_1, \xi_2) \ k_z / k_y \quad \hbox{if $x = z$} \qquad \hbox{and} \qquad
 b_z (x, y) \ = \ 0 \quad \hbox{if $x \neq z$}.
\end{equation}
By using \eqref{ICRP-bound}, we obtain that for any site $z \in \Lambda_n \subset \Lambda_m$,
 $$ \begin{array}{rcl}
     \Omega_{n, m}\super 2 q_z (\xi_1, \xi_2) & = &
    \displaystyle 2 d \, \sum_{y \in \Lambda_n} \ p (z, y) \ b_z (z, y) \vspace{5pt} \\ 
& + & \displaystyle 2 d \,  \sum_{y \in \Lambda_m \setminus \Lambda_n} \ p (z, y) \
      a_2 (z, y) \ [q_z (\xi_1, \xi_2 \super{z, +}) - q_z (\xi_1, \xi_2)] \vspace{5pt} \\ 
  & \leq & \displaystyle 2 d \, \lambda  \, \sum_{y \in \Lambda_n} \ p (z, y) \ q_y (\xi_1, \xi_2) \ k_z / k_y \vspace{5pt} \\ & + &
    \displaystyle 2 d \, \lambda  \sum_{y \in \Lambda_m \setminus \Lambda_n} \ p (z, y) \ q_y (\xi_2) \ k_z / k_y. \end{array} $$
 Now, assume that $\xi_1 (z) > \xi_2 (z)$ for some $z \in \Lambda_n$.
 From \eqref{ICRP-q1} and \eqref{ICRP-q2}, it follows that
 $$ \begin{array}{l}
  (\hat \Omega_{n, m}\super{1,1} + \hat \Omega_{n, m}\super {1,2}) \,q_z (\xi_1, \xi_2) \ = \
   \phi \ (\xi_1 (z) - \xi_2 (z)) \ k_z \vspace{5pt} \\ 
\hspace{100pt} - \
  (\xi_1 (z) - \xi_2 (z)) \ k_z \ = \ (\phi - 1) \,q_z (\xi_1, \xi_2) \end{array} $$
The same holds when $\xi_1 (z) < \xi_2 (z)$. In particular,
$$ (\hat \Omega_{n, m}\super{1,1} + \hat \Omega_{n, m}\super{1,2}) \,q_z (\xi_1, \xi_2) \ \leq \ 
(\phi - 1) \,q_z (\xi_1, \xi_2) $$
in any case since both members of the inequality are equal to 0 when $\xi_1 (z) = \xi_2 (z)$.
Finally, by observing that 
$$q_z (\xi_1 \super{x}, \xi_2 \super{x}) \ = \
\left\{\begin{array}{ll} 
q_z (\xi_1, \xi_2)    & \hbox{when} \ x \neq z \vspace{5pt} \\
0                     & \hbox{when} \ x = z,\end{array} \right.$$
we have
$$ \bar \Omega_{n, m}\super{1,2} q_z (\eta_1, \eta_2) \ = \ -  \,q_z (\eta_1, \eta_2). $$
Putting things together, we get the upper bound
 \begin{equation}\label{ICRP-lip} \hat \Omega_{n, m}  q_z (\xi_1, \xi_2) \ \leq \
\sum_{y \in \Lambda_n} \ c_{yz} \ q_y (\xi_1, \xi_2) \ + \ \sum_{y \in \Lambda_m \setminus\Lambda_n} \ g_{yz} \ q_y (\xi_2) 
\end{equation}
 where the coefficients $c_{yz}$ and $g_{yz}$ are given by
 $$ c_{yz} \ = \ \left\{\begin{array}{ll}
    \phi-1 & \hbox{if $y = z$} \vspace{5pt} \\
    \tilde {\lambda} \,p (z, y) \,k_z / k_y & \hbox{if $y \neq z$} \end{array} \right. \qquad \hbox{and} \qquad
    g_{yz} \ = \tilde{\lambda}  \, p (z, y) \,k_z / k_y. $$
Inequality \eqref{ICRP-lip} also holds for the coupled generator $\bar\Omega_{n,m}$.  Condition \eqref{ICRP-cond3} of Theorem \ref{ICRP-Chen2} is then satisfied.  By \eqref{ICRP-sequence} and \eqref{ICRP-controlled}, for any site $y\in \Lambda_m$ and any constant $M>1$,
we have  
 $$\displaystyle \sum_{z\in \Lambda_n} c_{yz}\ \leq \ \phi -1 \, +\,  2\,d\, \lambda \,  \sum_{z\in\Lambda_n}\  p(z,y)\
\displaystyle {k_z\over k_y}\ \leq\  \phi-1 + 2\ d \ \lambda\ M.$$
In particular, condition \eqref{ICRP-cond4} in Theorem \ref{ICRP-Chen2} holds whenever $\phi + 2\ d \ \lambda <1$.
In other respects conditions \eqref{ICRP-cond44} and \eqref{ICRP-cond5} are trivial. This completes the proof of Theorem \ref{ICRP-tauxfaible}.\\

\noindent{\slshape{Proof of Theorem \ref{IR-CR-phi}.}} We start by proving the first statement, i.e., if $\lambda> \lambda_c$ an epidemic may occur.  First, we note that by using basic coupling, the probability of an epidemic is nondecreasing with respect to the initial configuration and to each of the parameters $\lambda$, $\beta$, and $\phi$.
\begin{lemma}
\label{ICRP-monotonicity1}
The IRP($\infty$) and CRP($\infty$) are attractive and monotone with respect to the parameters $\lambda$, $\beta$ and $\phi$.\end{lemma}
Again by using basic coupling of the basic contact 
process with parameter $\lambda$ and the IRP($\infty$) (respectively,  CRP($\infty$)) with parameter $(\lambda,\beta,\phi)$ we show that both IRP($\infty$) and CRP($\infty$) have more infected individuals than the contact process.  This together,  with 
Lemma \ref{ICRP-monotonicity1},  implies that, when $\lambda> \lambda_c$, an epidemic may occur for any $\phi\geq 0$ and $\beta\geq 0$ for both processes.
To prove the second statement, we  will show that there exist $\phi_1>0$ and $\phi_2<\infty $ 
such that  if $\phi<\phi_1$ there is no epidemic, while  if $\phi>\phi_2$ an epidemic may occur. Due to the monotonicity with respect to the within infection rate $\phi$, this will imply the 
existence of  $\phi_c\in[\phi_1,\phi_2]$ such that Theorem \ref{IR-CR-phi} holds.

\indent  The existence of $\phi_1$ follows from the fact that when $\phi + 2 d (\lambda\vee\beta) < 1$, the CRP($\infty$) and IRP($\infty$) converge to the ``all 0'' configuration.
In others words, if \eqref{ICRP-TF} holds then there is not epidemic for the IRP($\infty$) and the CRP($\infty$). We  now deal with the existence of $\phi_2$.  
Let $\xi_t\super 1$  denote the IRP($\infty$) with parameters 
$(\lambda,0,\phi)$ and $\xi_t\super 2$ denote the IRP($\kappa$) with parameters 
$(\lambda,0,\phi)$. Using basic coupling we prove that if 
$\xi_0\super 1 (x) \geq \xi_0\super 2 (x)$ for any $ x\in\Z$ at time $0$, then
$\xi_t \super{1}$ and $\xi_t \super{2}$ can be constructed on  the same probability 
space in such way that 
$$ P_{(\xi_0 \super{1}, \xi_0 \super{2})} (\forall \ x \in \Z, \ \xi_t \super{1} (x) 
 \geq \xi_t \super{2} (x)) \ = \ 1,$$
where $P_{(\xi_0 \super{1}, \xi_0 \super{2})}$ is the law of the coupled process starting from $(\xi_0 \super{1}, \xi_0 \super{2})$. It follows that the IRP($\infty$) with parameters $(\lambda_1,0,\phi_1)$ has more infected individuals than the IRP($\kappa$) with parameters $(\lambda_2,0,\phi_2)$. It follows that
by Theorem 3, Belhadji and Lanchier (2006), an epidemic may occur for the IRP($\kappa$) with parameters  ($\lambda,\beta,\phi$) for all $\kappa >1$, provided the within infection rate $\phi$ is greater than some critical value. By Lemma \ref{ICRP-monotonicity1} the existence of $\phi_2$  such that an epidemic may occur for the IRP($\infty$) with parameters $(\lambda,\beta, \phi)$, for all $\phi\geq \phi_2$ and $\beta\leq \lambda$ follows. 
In the same way, that is by basic coupling we obtain that the CRP($\infty$) with parameters $(\lambda,\beta, \phi)$ has more infected individuals than the CRP($\kappa$) with parameters $(\lambda,0,\phi)$. Due to monotonicity of the CRP($\kappa$) with respect to $\kappa$,  
we can fix $\kappa$ such that $\kappa\lambda> \lambda_c$, and apply Theorem 1, Schinazi (2002) and Lemma \ref{ICRP-monotonicity1},  to get  the existence of  $\phi_2$  such that the CRP($\infty$) with  parameters $(\lambda,\beta,\phi)$ is not ergodic for all $\phi\geq \phi_2$ and $\beta\geq 0$. 
Thus Theorem \ref{IR-CR-phi} follows.

\section{Proof of theorems \ref{hydro-hth} and \ref{hydro-hthiv}}
\label{hydro-Sec: hydro}

\subsection{One block estimate.}
\label{hydro-Sec: oneblock}

\indent  It allows the replacement of a local function $h(\eta(x),\xi(x)), \, x\in \Z$ by a function of 
the {\it empirical density} $\eta^k(x)$ (respectively, $\xi^k(x)$) of healthy  (respectively, infected) individuals in a box of length $2\, k + 1$, $k\in\N$ centered at $x$:
 \begin{equation} \label{hydro-emdensity}
\eta^k (x) \ = \  \displaystyle {1 \over 2 k + 1} \ \displaystyle \sum_{|y-x| \leq k} \ \eta (y) \qquad \hbox{and} \qquad
\xi^k (x) \ = \  \displaystyle  {1 \over 2 k + 1} \ \displaystyle \sum_{|y-x| \leq k} \ \xi (y).\end{equation}

\begin{propo}\label{oneblock}[One block estimate](Mourragui, 1996) 
 Let $h$ be a bounded function on $\N \times \N$.  We have 
 $$ \displaystyle \lim_{k \to \infty} \ \lim_{N \to \infty} \displaystyle  \int_{\N^{\T_N}} \,\left ( {1 \over N}
 \sum_{x = 0}^{N - 1} \int_0^T V_k (\eta_s (x), \xi_s (x)) \,ds \right ) \ d\mu^N \ = \ 0, $$
 where
 $$ V_k (\eta(x),\xi(x)) \ = \ \left|\,  \displaystyle  {1 \over 2 k + 1} \
 \displaystyle \sum_{|x-y| \leq k } \ (h (\eta (y), \xi (y)) - {\tilde h} 
\,(\eta^k (x), \,\xi^k (x)) \,\right|. $$
\end{propo}
We refer to Mourragui (1996) for its proof.

\pagebreak

\indent In what follows, we will use intensively the following change of variables formulas stated for each cylinder function on $\N^{\T_N}\times \N^{\T_N}$. 
Let $\nu^N_a \times \nu^N_b$ be the product of Poisson measures on $\N^{\T_N}\times \N^{\T_N}$, with arbitrary parameters $a > 0$ and $b > 0$ respectively. We have:
\begin{equation}\label{cv} \begin{array}{rcl} \vspace{3pt}
\noindent \displaystyle\int f(\eta^{x,-}, \xi^{x,+}) \,  d \,(\nu^N_a \times \nu^N_b) (\eta,\xi) &  = & \displaystyle {a\over b}\, 
\int {\xi(x) \over 1+ \eta(x)} f(\eta,\xi) \, d \, (\nu^N_a \times \nu^N_b) (\eta,\xi),  \\ \vspace{3pt}
\noindent \displaystyle \int f(\eta^{x,+}, \xi^{x,-})\, d\, (\nu^N_a \times \nu^N_b) (\eta,\xi)  & = & \displaystyle {b\over a}\,
\int {\eta(x)\over 1+\xi(x)} \, f(\eta,\xi) \, d\, (\nu^N_a \times \nu^N_b) (\eta,\xi). 
\end{array}
\end{equation}

\noindent We will use  the entropy inequality: If $\alpha^N$ and $\beta^N$ are two measures on $\N^{\T_N}\times \N^{\T_N}$, then  for all bounded function $U$ and $\alpha > 0$, 
\begin{equation}\label{intro-entropique}
\int U \, d\,\beta^N \leq {1\over \alpha }\log \int \exp(\alpha\, U) \, d\, \alpha^N + 
{1\over \alpha}\,  \Hh\, \left[\beta^N|  \alpha^N\right ].
\end{equation}

\subsection{Proof of Theorem \ref{hydro-hth}}
\label{ICRP-Sec: proof3}

It is divided in several lemmas.  The objective is to prove that 
 \begin{equation}\label{zero}
\displaystyle \lim_{N \rightarrow + \infty}\mu_t^N(A^{G_1,G_2,\delta}_{N})=0,
  \end{equation}
where
\begin{equation}\label{hydro-set}A^{G_1,G_2,\delta}_{N} \ =\  \left  \{(\eta,\xi): \left | \displaystyle  {1\over N} \displaystyle \sum_{x=0}^{N-1} \eta(x) \, G_1(x/N) 
 -\int_0^1 G_1(\theta)\lambda_1(t,\theta) d\theta \right | >\delta \right.  \end{equation}
 $$\hspace{6cm}\hbox {and}\left.  \quad \left | \displaystyle  {1\over N} \sum_{x=0}^{N-1} \xi(x)\, 
G_2(x/N) -\int_0^1 G_2(\theta)\lambda_2(t,\theta) d\theta \right | > \delta \right \},$$
and $(\lambda_1(t,.),\lambda_2(t,.))$ is the solution of \eqref{hydro-PDE}. 
The entropy inequality \eqref{intro-entropique}  allows us to write
\begin{equation}\label{hydro-zero2}\mu_t^N(A^{G_1,G_2,\delta}_{N,t})\leq {  \displaystyle {1\over N} \log2+ {1\over N}\,  \Hh \, \left [\mu^N_t| \nu_{\lambda_1(t,.)}^N\times \nu_{\lambda_2(t,.)}^N\right ]\over\displaystyle 
{1\over N} \log\left [1+\{\nu_{\lambda_1(t,.)}^N\times \nu_{\lambda_2(t,.)}^N(A^{G_1,G_2,\delta}_{N})\}^{-1} \right ] }.
\end{equation}  

Proof of \eqref{zero} relies on  the following steps.

\begin{propo}\label{b-entropy}
For each $t$ in $[0,T]$, there exists a function $A_N^t$ which converges to zero when $N$ goes to infinity and a constant $C$ such that
\begin{equation}\label{hydro-upperbound}
{1\over N}\, \Hh \, \left [ \mu^N_t| \nu_{\lambda_1(t,.)}^N\times \nu_{\lambda_2(t,.)}^N\right ]\leq A_N^t + {C\over N}\int_0^t\,  \Hh \,
\left [ \mu^N_s|\nu_{\lambda_1(s,.)}^N\times \nu_{\lambda_2(s,.)}^N\right ]ds.
\end{equation}
Using Varadhan theorem (Chen, 1992 page 286), we have for all profiles $\rho_1(.)$ and $\rho_2(.)$ and $\delta > 0$ 
\begin{equation}\label{zero3} \lim_{N\longrightarrow\infty}  \displaystyle {1\over N}\, \log \, \Big (\nu_{\rho_1(t,.)}^N\times \nu_{\rho_2(t,.)}^N\Big ) (A^{G_1,G_2,\delta}_{N,t}) \ < \ 0.
\end{equation}
\end{propo}
By using \eqref{hydro-upperbound}  and applying Gronwall lemma we then prove that:
\begin{equation}\label{zero1}
\lim_{N\longrightarrow +\infty} \displaystyle {1\over N}\,  \Hh\,  \left [ \mu^N_t| \nu_{\lambda_1(t,.)}^N\times \nu_{\lambda_2(t,.)}^N\right ]= 0.
\end{equation}
Inequality \eqref {hydro-zero2},  \eqref{zero1} and \eqref{zero3} imply \eqref{zero}. \\

\indent In the proof of the proposition \ref{b-entropy}, we will need to write $\lambda_1(t,.)^{-1}$ and $\lambda_2(t,.)^{-2}$, and to avoid technical difficulties,  we  assume that $(\lambda_1(t,.),\lambda_2(t,.))$ the solution of \eqref{hydro-PDE} is bounded below by a strictly positive constant $K$:
$$\inf_{t\geq 0} \ \inf_{x\in\T_N} \lambda_i(t,x/N) \, = \, K , \quad  \hbox{for} \ i=1,2.$$
Indeed, if it is not the case, the proof may be modified by replacing  $\lambda_1(t,.)$ and $\lambda_2(t,.)$ by $\lambda_1(t,.)+ \varepsilon$ and $\lambda_2(t,.)+ \varepsilon$ ($\varepsilon>0$) and by letting $\varepsilon$ goes to zero.  In order to prove the proposition we need also to compute the relative entropy $\Hh \, \left [ \mu^N_t| \nu_{\lambda_1(t,.)}^N\times \nu_{\lambda_2(t,.)}^N\right ]$. 

\indent We denote by  $f^N$ and $f^N_t$ the Radon-Nikodym derivatives of $\mu^N$ and $\mu^N_t= \mu^NS^N_t$ with respect to the reference measure $(\nu_{\rho}^N\times\nu_{\rho}^N)$. Let $\psi^N_t$ denote the Radon-Nikodym derivative of $\nu^N_{\lambda_1(t,.)}\times \nu^N_{\lambda_2(t,.)}$ with respect to the reference measure. Because $\nu^N_{\lambda_1(t,.)}\times \nu^N_{\lambda_2(t,.)}$ and $\nu^N_{\rho}\times \nu^N_{\rho}$ 
are product measures, $\psi^N_t$ can be computed explicitly 
\begin{equation}\label{psi}
\begin{matrix}  \psi^N_t(\eta,\xi)=  & \exp\left (\displaystyle \sum_{i=0}^{N-1} \left \{ \eta(i) \log \left ({\lambda_1(t,i/N)\over \rho}\right ) + \rho  -\lambda_1(t,i/N)\right \} \right )\hfill\cr\cr
& \times \exp\left (\displaystyle \sum_{i=0}^{N-1}\left \{ \xi(i) \log \left ({\lambda_2(t,i/N)\over
\rho}\right )+\rho -\lambda_2(t,i/N)\right \} \right).
\end{matrix}
\end{equation}

\begin{proof} \ (Proposition \ref{b-entropy}) \
We  derivate  the relative entropy, using  that the density $f^N_t$ is the solution of the Kolmogorov forward equation $\partial_t f^N_t= \Omega_N^*f^N_t$. 
\begin{equation}\label{tropi}\begin{array}{rcll}\vspace{8pt}
\displaystyle {d\over dt} \, \Hh \, \left [ \mu^N_t| \nu_{\lambda_1(t,.)}^N\times \nu_{\lambda_2(t,.)}^N\right ]  &=& \displaystyle{d\over dt}\int f^N_t \, \log \Big ({f^N_t\over \psi^N_t}\Big )\,  d(\nu^N_{\rho}\times \nu^N_{\rho}) \\ \vspace{8pt}
& = & \displaystyle \int  f^N_t  \Omega_N \log\Big ({f^N_t\over \psi^N_t}\Big ) \, d (\nu^N_{\rho}\times \nu^N_{\rho}) -  \displaystyle \int  {f^N_t\over \psi^N_t}  \displaystyle {d\over dt}( \psi^N_t) \, d (\nu^N_{\rho}\times \nu^N_{\rho})  \\ \vspace{8pt}
& = & N^2 \displaystyle \int f^N_t\, \Omega_N\super{{\cal D},1} \log\Big({f^N_t\over \psi^N_t}\Big)\,  d (\nu^N_{\rho}\times \nu^N_{\rho}) \hfill\\ \vspace{8pt}
& & \hspace{3.3cm} + \displaystyle  N^2\, \int f^N_t\,  \Omega_N\super{{\cal D},2} \log\Big ({f^N_t\over \psi^N_t}\Big ) \, d(\nu^N_{\rho}\times \nu^N_{\rho})\hfill \\ \vspace{8pt}
&+&  \displaystyle \int f^N_t \Omega_N\super {\cal R} \log\Big ({f^N_t\over \psi^N_t}\Big )\, d(\nu^N_{\rho}\times \nu^N_{\rho}) 
- \displaystyle \int \displaystyle {f^N_t\over \psi^N_t}\displaystyle  {d\over dt}  (\psi^N_t) \, 
d( \nu^N_{\rho}\times \nu^N_{\rho}) \hfill\\ \vspace{8pt}
 & = &  \ I_1 \, + \,  I_2\, + \, I_3\, - \, I_4.
\end{array}
\end{equation}

\indent To compute $I_1$ and $I_2$ we use the explicit expression for $\psi^N_t$ given in \eqref{psi}, the fact that $\Omega_N\super {{\cal D},1}$ and $\Omega_N\super {{\cal D},2}$  are  self-adjoint with respect to the product measure $\nu^N_{\rho}\times \nu^N_{\rho}$, and    
\begin{equation} \label{elementary}
x\, \big [\log(y)-\log(x)\big ]\, \leq \, y-x, \quad \hbox{for all}\  x,y >0,
\end{equation}
\begin{equation}
\displaystyle \sum_{i=0}^{N-1} \big [\lambda_1(t,(i+ 1)/ N)+ \lambda_1(t,(i-1)/ N)-2\lambda_1(t,i/ N) 
\big ] = 0.
\end{equation} 
We obtain that 
$$\begin{array}{rcl}\vspace{3pt}
I_1 & \leq  & N^2\displaystyle \int \displaystyle {f^N_t\over \psi^N_t} \, \Omega_N\super{{\cal D},1}  (\psi^N_t)\,  d\,(\nu^N_{\rho}\times \nu^N_{\rho}) \\ \vspace{3pt}
& \leq &  \displaystyle  {N^2\over 2}\displaystyle \sum_{|j-i|=1}\int \eta(i) \left
({\lambda_1(t,j/N)\over \lambda_1(t, i/N)}
 -1 \right )\, d \, \mu^N_t(\eta,\xi)\\ \vspace{3pt}
& = & \hspace{-3pt}\displaystyle {N^2\over 2}\displaystyle \sum_{i=0}^{N-1}\int \left 
({\eta(i)\over \lambda_1(t,i/N)}-1 \right ) \hspace{-3pt}
\Big [\lambda_1(t,(i+1)/ N)+ \lambda_1(t,(i-1)/ N)  -2\lambda_1(t,i/ N)  \Big ]  d\, \mu^N_t(\eta,\xi) 
\end{array}$$
\pagebreak

To take advantage of the fact that   $(\lambda_1(t,.),\lambda_2(t,.))$ is  solution of \eqref{hydro-PDE} and to conjure up the Laplacian of $(\lambda_1(t,.),\lambda_2(t,.))$ that will appear later in $I_4$ with negative sign, we observe that a Taylor-Young expansion gives:
$$ N^2 \, \Big [\lambda_1(t,(i+1)/ N)+ \lambda_1(t,(i-1)/ N)  -2\lambda_1(t,i/ N)  \Big ] = {\partial^2\over \partial \theta^2}\lambda_1(t,i/ N) + o(1/N^2).$$ 
So,
$$\begin{array}{rcl}\vspace{3pt}
I_1 \ \leq  \ \displaystyle {1\over 2}\displaystyle \sum_{i=0}^{N-1}\int \left ({\eta (i)\over \lambda_1(t,i/N)}-1 \right )  {\partial^2\over \partial \theta^2}\lambda_1(t,i/ N) \, d\, \mu^N_t(\eta,\xi) \, + \, o(1/N). 
\end{array}$$
The second term $I_2$ has a similar upper bound:
$$ I_2 \  \leq \ \displaystyle {1\over 2}\displaystyle \sum_{i=0}^{N-1}\int \left ({\xi (i)\over \lambda_2(t,i/N)}-1 \right )  {\partial^2\over \partial \theta^2}\lambda_2(t,i/ N) \, d\, \mu^N_t(\eta,\xi) \, + \, o(1/N).$$ 
To deal with the third term $I_3$, we apply  inequality \eqref{elementary},  
and by using substitution rule \eqref{cv}  we get:
\begin{eqnarray*}\vspace{2pt}
I_3 & \leq & \displaystyle   \sum_{i=0}^{N-1} \int  \Big [ {\eta(i)\over \lambda_1(t,i/N)}\, \beta_1(\eta(i)-1, \xi(i)) - \beta_1(\eta(i),\xi(i)) \Big ] \,
d\mu^N_t(\eta,\xi) \\ \vspace{2pt}
& + &  \displaystyle   \sum_{i=0}^{N-1} \int \Big [ { \lambda_1(t,i/N)\over \eta(i) + 1}\, \delta_1(\eta(i)+1, \xi(i)) - \delta_1(\eta(i),\xi(i)) \Big ]\,
d\mu^N_t(\eta,\xi) \\ \vspace{2pt}
& + &  \displaystyle   \sum_{i=0}^{N-1} \int \Big [ {\xi(i)\over  \lambda_2(t,i/N)}\, \beta_2(\eta(i), \xi(i)-1) - \beta_2(\eta(i),\xi(i))  \Big ] \,
d\mu^N_t(\eta,\xi) \\ \vspace{2pt}
& + &  \displaystyle   \sum_{i=0}^{N-1} \int \Big [ { \lambda_1(t,i/N)\over \xi(i) + 1} \, \delta_2(\eta(i), \xi(i)+1) - \delta_2(\eta(i),\xi(i))\Big ]
d\mu^N_t(\eta,\xi)\\ \vspace{2pt}
& + & \phi  \displaystyle   \sum_{i=0}^{N-1} \int \Big [ { \lambda_1(t,i/N)\over \lambda_2(t,i/N)} \times {\xi(i)(\xi(i)-1)\over \eta(i)+1} - \xi(i)\ind_{\{\eta(i)>0\}} \Big ]\, d\mu^N_t(\eta,\xi)\\ \vspace{2pt}
& + &  \displaystyle   \sum_{i=0}^{N-1} \int \left (  { \lambda_1(t,i/N)\over \lambda_2(t, i/N) }\times  {\xi(i) \over \eta(i)+ 1}\ind_{\{ \xi(i)=1\}} - \ind_{\{\eta(i)>0,\, \xi(i)=0\}} \right )  \times \Big (\lambda\displaystyle \sum_{j\sim i}\xi(j) \Big ) \,
d\mu^N_t(\eta,\xi)\\ \vspace{2pt}
& + &  \displaystyle   \sum_{i=0}^{N-1} \int \left (  { \lambda_1(t,i/N)\over \lambda_2(t, i/N) }\times  {\xi(i) \over \eta(i)+ 1}- \ind_{\{\eta(i)>0,\, \xi(i)>0\}} \right )  \times \Big (\beta\displaystyle \sum_{j\sim i}\xi(j) \Big ) \,
d\mu^N_t(\eta,\xi)\\ \vspace{2pt}
& + &  \displaystyle   \sum_{i=0}^{N-1} \int \Big [  { \lambda_2(t,i/N)\over \lambda_1(t, i/N) }  \eta(i)  - \xi(i) \Big ] \, d\mu^N_t(\eta,\xi).
\end{eqnarray*}

We rewrite the fourth term $I_4$ using  that $(\lambda_1(t,.),\lambda_2(t,.))$ solves  equation \eqref{hydro-PDE}
\begin{eqnarray*}\vspace{4pt}
I_4 & = &  \displaystyle \sum_{i=0}^{N-1} \int \Big [  \left ({\eta(i)\over \lambda_1(t,i/ N)} -1\right ) {d\over dt } \lambda_1(t,i/N ) 
+ \left ({\xi(i)\over \lambda_2(t,i/ N)} -1\right ) {d\over dt } \lambda_2(t,i/ N )\Big ]\, d\mu^N_t(\eta,\xi) \\ \vspace{4pt}
& = &  \displaystyle \sum_{i=0}^{N-1}\int  \left ({\eta(i)\over \lambda_1(t,i/ N)} -1\right ) \left (\displaystyle {1\over 2} 
{\partial^2\over \partial \theta^2 }\lambda_1(t,i/ N) +
 {\widetilde \beta}_1(\lambda_1(t,i/ N,\lambda_2(t,i/N))\, -\right. \\ \vspace{4pt}
   &  &  \hspace{3cm}  {\widetilde \delta}_1(\lambda_1(t,i/ N), \lambda_2(t,i/ N)) 
+\left.  {\widetilde g}(\lambda_1(t,i/ N),\lambda_2(t,i/ N))\right ) \, d\mu^N_t(\eta,\xi)
\\ \vspace{4pt}
& +&   \displaystyle \sum_{i=0}^{N-1} \int \left ({\xi(i)\over
  \lambda_2(t,i/ N)} -1\right ) \Big ( \displaystyle
 {1\over 2} {\partial^2\over \partial \theta^2 }\lambda_2(t,i/ N) + 
{\widetilde \beta}_2(\lambda_1(t,i/ N),\lambda_2(t,i/ N)) \, - \\ \vspace{4pt}
  &  &  \hspace{3cm}    {\widetilde \delta}_2(\lambda_1(t,i/N), \lambda_2(t,i/ N)) 
    -  {\widetilde g}(\lambda_1(t,i/ N),\lambda_2(t,i/ N)) \Big )\, d\mu^N_t(\eta,\xi).
\end{eqnarray*}
Since the birth and death rates are not bounded, we have to truncate them with
indicators of sets like $A_M=\{\eta(i) + \xi(i) \leq M\}$. To control terms 
with $\{\eta(i) \geq M\}$ or $\{ \xi(i) \geq M\}$, we need  the 
\begin{lemma}\label{phi}
 Let $\varphi$ be a function on $\N\times \N$ such that
\begin{equation}\label{hydro-phicond}
\lim_{k_1\rightarrow +\infty} {\varphi(k_1,k_2)\over \delta_1(k_1,k_2)}\, =\,  0. 
\end{equation}
Then,
\begin{equation}\label{hydro-limzero}\lim_{M\rightarrow\infty}\,\limsup_{N\rightarrow\infty}\,{1\over N}\, \sum_{x=0}^{N-1}\int_0^T \int \varphi(\eta(x),\xi(x))\, \ind_{\{\eta(x)\, >\,  M\}}\, f^N_s(\eta,\xi)\, d\, (\nu_{\rho}^N\times \nu_{\rho}^N ) \,(\eta,\xi)\, ds \, = \,  0.  
\end{equation}
Let $\varphi$ be a function on $\N\times \N$ such that $\displaystyle \lim_{k_2\rightarrow +\infty} {\varphi(k_1,k_2)\over \delta_2(k_1,k_2)} = 0$, then 
\begin{equation}\label{hydro-limzero1}
\displaystyle \lim_{M\rightarrow\infty}\,\limsup_{N\rightarrow\infty}\,{1\over N}\, \sum_{x=0}^{N-1}\int_0^T \int \varphi(\eta(x),\xi(x))\, \ind_{\{\xi(x)\, >\,  M\}}\, f^N_s(\eta,\xi)\, d\, (\nu_{\rho}^N\times \nu_{\rho}^N) \,(\eta,\xi)\, ds  =   0. 
\end{equation}
\end{lemma}
{\scshape Proof.}\hspace{2mm}  We will use a martingale argument. By \eqref{hydro-phicond}, for all $\varepsilon > 0$ 
there exists $M_1\in\N$ such that, for all $M\geq M_1$
$$\varphi(\eta(x),\xi(x))\, \ind_{\{\eta(x)\, > \, M\}} \,\leq \, {\varepsilon\over 2}\, \delta_1(\eta(x),\xi(x))\, \ind_{\{\eta(x)\, > \, M\}}. $$ 
Moreover, by the explicit formulas for $\beta_1$ and $\delta_1$, given in \eqref{Taux},  
it follows that there exists $C>0$ such that for $\varepsilon >0$ and $M>M_1$
\begin{equation}\label{hydro-boundphi}
\varphi(\eta(x),\xi(x))\, \ind_{\{\eta(x)\, > \, M\}} \,\leq \, \varepsilon\,\Big[ \delta_1(\eta(x),\xi(x)) - \beta_1(\eta(x),\xi(x)) - g(\eta(x),\xi(x)) + C\Big ]. 
\end{equation}
We have the following centered martingale with respect to the  filtration 
${\cal F}_t=\sigma\{(\eta_s,\xi_s); s\leq t\}$ 
\begin{eqnarray*}
M^N_t & = & \displaystyle \sum_{x=0}^{N-1}\eta_t(x) - \sum_{x=0}^{N-1}\eta_0(x)- \int_0^t \Omega_N\super {\cal R}\left (\sum_{x=0}^{N-1} \eta_s(x)\right )\, ds\\
& = &  \displaystyle \sum_{x=0}^{N-1}\eta_t(x)- \sum_{x=0}^{N-1}\eta_0 (x)\\
&+& \displaystyle\sum_{x=0}^{N-1} \int_0^t\Big [ \delta_1 (\eta_s(x),\xi_s(x))- \beta_1(\eta_s(x),\xi_s(x)) - g(\eta_s(x),\xi_s(x))\Big ] ds.  
 \end{eqnarray*}
Because $M^N_t$ is centered, by the entropy inequality \eqref{intro-entropique}  and 
by \eqref{hydro-boundphi},  we obtain  for $M$ large enough and  $\varepsilon$ small:
\begin{eqnarray*} 
& &   {1\over N}\, \sum_{x=0}^{N-1}\int_0^t \int \varphi(\eta(x),\xi(x))\, \ind_{\{\eta(x)\, >\,  M\}}\, f^N_s(\eta,\xi)\, d\, (\nu_{\rho}^N\times \nu_{\rho}^N ) \,(\eta,\xi)\, ds  \hfill \cr
& & \leq  {\varepsilon \over N} \sum_{x=0}^{N-1} \int_0^t \int \Big [\,  \delta_1(\eta(x),\xi(x)) - \beta_1(\eta(x),\xi(x))- g(\eta(x),\xi(x)) + C \, \Big ] \,  d\mu^N_s (\eta,\xi)\,  ds. \hfill \cr
& & \leq  t\, \varepsilon \displaystyle {1\over N}\sum_{x=0}^{N-1}\left(C- \int \eta(x) f^N_t(\eta,\xi)\, d(\nu^N_{\rho}\times \nu^N_{\rho})(\eta,\xi) + \int \eta(x) f^N(\eta,\xi) \, d( \nu^N_{\rho}\times \nu^N_{\rho})(\eta,\xi)\right)\hfill\cr
& &\leq  \varepsilon\,  C_t,\hfill
\end{eqnarray*}
therefore \eqref{hydro-limzero} follows. 

\indent The computation to prove \eqref{hydro-limzero1} is quite different; for all $\varepsilon > 0$ we have  
\begin{eqnarray*} 
& &   \varepsilon  \sum_{x=0}^{N-1} \Big [\beta_2(\eta(x),\xi(x)) + \lambda\, \ind_{\{\eta(x) >  0, \xi(x) = 0\}}\, \Big (\sum_{|y-x|=1} \xi(y) \Big ) + \phi \ind_{\{\eta(x) > M\}}\, \xi(x)\Big ] \hfill \cr
& & \leq  {\varepsilon\over 2}\, \sum_{x=0}^{N-1} \Big [ \delta_2(\eta(x),\xi(x)) + 2 \, C\Big ],\hfill
\end{eqnarray*}
it follows that there exists $C> 0$ such that for all $\varepsilon > 0$, 
$$\sum_{x=0}^{N-1} \Big [{\varepsilon\over 2}\, \delta_2(\eta(x),\xi(x)) - \varepsilon \, \beta_2(\eta(x),\xi(x)) + \varepsilon \, g(\eta(x),\xi(x)) + \varepsilon\, C\Big ] > 0.$$
Finally we obtain the result by the following inequality
$$ \sum_{x=0}^{N-1} \, \varphi(\eta(x),\xi(x))\ind_{\{\xi(x) > M\}} \leq \varepsilon\, 
\sum_{x=0}^{N-1} \Big[ \delta_2(\eta(x),\xi(x)) -\beta_2(\eta(x),\xi(x)) -g(\eta(x),\xi(x))+ C\Big ].\quad \square$$

\indent Let us now integrate \eqref {tropi}, putting things together, and removing the negative terms:
 \begin{equation}\label{hydro-uperbound}\begin{array}{lll}
\displaystyle  {1\over N} \Hh \, \left  [ \mu^N_t|\nu_{\lambda_1(t,.)}^N\times \nu_{\lambda_2(t,.)}^N\right]  & \leq &   
\displaystyle {1\over N}\, \Hh\, \left [ \mu^N| \nu_{m_1(.)}^N\times \nu_{m_2(.)}^N\right ] + F(M,N,T) + o \Big ( {1\over N} \Big ) \\ \\
&& + \displaystyle {1\over N} \displaystyle \sum_{i=0}^{N-1}\int_0^t \int 
\Big \{ \sum_{i=1}^7 T_i \Big\} d\, \mu_s^N(\eta,\xi)\, ds.
\end{array}\end{equation}
In order to simplify the expression of $T_i, \, (i=1\ldots,7)$,  we set:
$$\beta_{1,M}(\eta(i),\xi(i))  \,=\,    \beta_1(\eta(i),\xi(i))\, 
\ind_{\{\eta(i)\, \leq \, M,\,  \xi(i)\, \leq \, M\}}$$ 
$$\varphi_{1,M}(\eta(i), \xi(i)) \, = \,   \eta(i)\, {\beta}_1(\eta(i)-1,\xi(i))\, 
\ind_{\{ \eta (i)\, \leq \, M+1,\,  \xi(i)\, \leq \, M \}},$$
$$\begin{array}{lll}
T_1 & =& \displaystyle {1\over \lambda_1(s,i/ N) }\,  \varphi_{1,M}(\eta(i), \xi(i))- \beta_{1,M}(\eta(i),\xi(i))-\\ \vspace{5pt}
& & \hspace{5cm}\left ( \displaystyle{\eta(i)\over \lambda_1(s,i/ N)}-1 \right ) 
{\widetilde \beta}_{1,M}(\lambda_1(s,i/  N),\lambda_2(s,i/ N)).
\end{array}$$

$$\delta_{1,M}(\eta(i),\xi(i)) \, = \, \delta_1(\eta(i),\xi(i))\, \ind_{\{\eta(i)\, \leq \,  M, \, \xi(i)\, \leq\,  M\}}$$ 
$$ \phi_{1,M}(\eta(i), \xi(i)) \, = \, {1\over \eta(i)+1} \, {\delta}_1(\eta(i)+1,\xi(i))\, 
\ind_{\{\eta(i)\leq M-1, \, \xi(i)\, \leq \, M\}},$$
and 
$$\begin{array}{lll}
T_2 &=& \lambda_1(s,i/ N) \phi_{1,M}(\eta(i),\xi(i))-\delta_{1,M}(\eta(i),\xi(i))+\\ \vspace{5pt}
& & \hspace{5cm} \left( \displaystyle{\eta(i)\over\lambda_1(s,i/  N)}-1 \right)
{\widetilde \delta}_{1,M}(\lambda_1(s,i/  N),\lambda_2(s,i/ N)),
\end{array}$$

$$\beta_{2,M}(\eta(i),\xi(i))  \,=\,    \beta_2(\eta(i),\xi(i))\,
 \ind_{\{\eta(i)\, \leq \, M,\,  \xi(i)\, \leq \, M\}}$$ 
$$\varphi_{2,M}(\eta(i), \xi(i)) \, = \,   \xi(i)\, {\beta}_2(\eta(i),\xi(i)-1)\, 
\ind_{\{ \eta (i)\, \leq \, M,\,  \xi(i)\, \leq \, M +1\}},$$
and 
\pagebreak

$$\begin{array}{lll}
T_3&=&   \,  \displaystyle{1\over \lambda_2(s,i/ N) }\,  \varphi_{2,M}(\eta(i), \xi(i))- \beta_{2,M}(\eta(i),\xi(i))- \\ \vspace{5pt}
& & \hspace{5.5cm} \left ( \displaystyle{\xi(i)\over \lambda_2(s,i/ N)}-1 \right ) {\widetilde \beta}_2(\lambda_1(s,i/  N),\lambda_2(s,i/ N)),
\end{array}$$

$$\delta_{2,M}(\eta(i),\xi(i)) \, = \, \delta_2(\eta(i),\xi(i))\, \ind_{\{\eta(i)\, \leq \,  M, \, \xi(i)\, \leq\,  M\}}$$
$$\phi_{2,M}(\eta(i), \xi(i)) \, = \, {1\over \xi(i)+1} \, {\delta}_2(\eta(i),\xi(i)+1)\, 
\ind_{\{\eta(i)\leq M, \, \xi(i)\, \leq \, M-1\}},$$
and 
$$\begin{array}{lll}
T_4 &=&   \,  \lambda_2(s,i/ N)\,  \phi_{2,M}(\eta(i), \xi(i))- \delta_{2,M}(\eta(i),\xi(i))+ 
\\ \vspace{5pt}
&& \hspace{5.5cm} \left ( \displaystyle{\xi(i)\over \lambda_2(s,i/ N)}-1 \right ) {\widetilde \delta}_2
(\lambda_1(s,i/ N), \lambda_2(s,i/ N)),
\end{array}$$

$$e_{1,M}(\eta(i),\xi(i))\, = \,  \phi \, \xi(i) \ind_{\{\eta(i)\, >\, 0, \, \xi(i)\, \leq \, M\}}$$ 
$$E_{1,M}(\eta(i), \xi(i)) \, = \, \phi\,  \displaystyle {\xi(i)\, (\xi(i)-1)\over \eta(i)+1} \, 
\ind_{\{ \xi(i)\, \leq\,  M +1\}},$$
and
$$\begin{array}{lll}\vspace{3pt}
T_5 &=&  \,  \displaystyle{\lambda_1(s,i/ N)\over \lambda_2(s,i/ N)} \, E_{1,M}(\eta(i), \xi(i))- 
e_{1,M}(\eta(i),\xi(i)) + 
\left (  \displaystyle {\eta(i)\over \lambda_1(s, i/ N)} -  \displaystyle{\xi(i)\over \lambda_2(s, i/ N)} \right ) \times \\ \vspace{5pt}
& & \hspace{8cm} {\tilde e}_{1,M}(\lambda_1(s,i/ N),\lambda_2(s,i/ N)),
\end{array}$$

$$e_{2,M}(\eta(i),\xi(i))  = \lambda\,  \ind_{\{\eta(i) >  0, \, \xi(i)=0\}} 
 \Big (\sum_{j\in \T_N} p(j,i)\, \xi(j)\ind_{\{\xi(i) \leq  M\}}\Big ), $$
$$E_{2,M}(\eta(i), \xi(i)) =    {\lambda \ind_{ \{\xi(i) =  1\}} \over \eta(i)+1}
 \Big (\sum_{j\in\T_N} p(j,i)\, \xi(j) \ind_{\{\xi(j) \leq  M\}}\Big),$$
and 
$$\begin{array}{lll}\vspace{3pt}
T_6&=&  \,  \displaystyle {\lambda_1(s,i/ N)\over \lambda_2(s, i/ N) } \, E_{2,M}(\eta(i), \xi(i))- e_{2,M}(\eta(i),\xi(i)) + \left ( \displaystyle {\eta(i)\over \lambda_1(s, i/ N)} - {\xi(i)\over \lambda_2(s, i/ N)} \right ) \times  \\ \vspace{5pt}
& & \hspace{8cm} {\tilde e}_{2,M}(\lambda_1(s,i/ N),\lambda_2(s,i/ N)).
\end{array}$$
$$r_M(\eta(i),\xi(i)) \, = \,  \xi(i) \, \ind_{\{\xi(i)\, \leq \, M, \, \eta(i)\, \leq\,  M\}}$$
$$R_{M}(\eta(i), \xi(i)) \, = \, \eta(i)  \, \ind_{\{ \xi(i)\, \leq \, M-1,\,  \eta(i)\, \leq \, M+1\}},
$$
and 
$$\begin{array}{lll}\vspace{3pt}
T_7 &=&  \,  \displaystyle{\lambda_2(s,i/ N)\over \lambda_1(s, i/ N) } \, R_M(\eta(i), \xi(i))- r_{M}(\eta(i),\xi(i)) -
\left ( \displaystyle{\eta(i)\over \lambda_1(s, i/ N)} -  \displaystyle{\xi(i)\over \lambda_2(s, i/ N)} \right )\times 
\\ \vspace{5pt} 
&& \hspace{6cm} {\tilde r}_{M}(\lambda_1(s,i/ N),\lambda_2(s,i/ N)),
\end{array}$$ 
Then $F(M,N,T)$ contains all terms with $\ind_{\{\eta(x) \, > \, M\}}$ and 
$\ind_{\{\xi(x) \, > \, M\}}$:
\begin{eqnarray*} \vspace{5pt}
& & F(M,N,T)  =   \displaystyle {1\over N} \sum_{i=0}^{N-1}\int_0^T \int 
f^N_s(\eta,\xi) \,\Big [  \\ \\ \vspace{5pt}
& & \Big ( {1\over \lambda_1(s,i/ N)}\, \varphi_1(\eta(i), \xi(i)) + \lambda_1(s,i/ N)\, 
\phi_1(\eta(i), \xi(i))\Big ) \times \Big (\ind_{\{\eta(i)\geq M\}} + \ind_{\{\xi(i)\geq M\}}\Big )  
\\ \\ \vspace{5pt}
&   + & \, \Big ({1\over \lambda_2(s,i/ N)}\, \varphi_2(\eta(i), \xi(i)) +  \lambda_2(s,i/ N)\, 
\phi_2(\eta(i), \xi(i))\Big ) \times \Big (\ind_{\{\eta(i)\geq M\}} + \ind_{\{\xi(i)\geq M\}}\Big ) 
 \\ \\ \vspace{5pt}
&  + & \, {\lambda_2(s,i/ N)\over \lambda_1(s, i/ N) } \, \eta(i) \, 
\Big (\ind_{\{\eta(i)\geq M\}} + \ind_{\{\xi(i)\geq M\}}\Big ) + {\eta(i)\over \lambda_1(s,i/ N)} \,{\tilde {\tilde {\delta}}}_{1,M}(\lambda_1(s,i/ N), \lambda_2(s, i/ N))\\  \\ \vspace{10pt}
&  + & \,  \phi\, { \lambda_1(s,i/ N)\over \lambda_2(s, i/ N) } \, { \xi(i)(\xi(i)-1)\over \eta(i)+1}\, \ind_{\{\xi(i)\geq M\}} + {\xi(i)\over \lambda_2(s,i/ N)} \,{\tilde {\tilde {\delta}}}_{2,M}(\lambda_1(s,i/ N), \lambda_2(s, i/ N)) \\ \\ \vspace{10pt} 
 &   + &  \, {\lambda_1(s,i/ N)\over \lambda_2(s, i/ N)}\, 
{\lambda\over \eta(i)+1}\ind_{\{\xi(i)=1\}} \Big(\sum_{|j-i|=1}\xi(j)\, \ind_{\{\xi(j)\, > \, M \}}\Big)
 +  \\ \\ \vspace{5pt}
&& \hspace{7cm}{\eta(i)\over \lambda_1(s,i/ N)} \,{\tilde {\tilde {r}}}_{M}(\lambda_1(s,i/ N), \lambda_2(s, i/ N))
 \\ \\ \vspace{5pt} 
&  + & \,  {\tilde {\tilde {\beta}}}_{1,M}(\lambda_1(s,i/ N), \lambda_2(s, i/ N)) 
+ {\tilde {\tilde {\beta}}}_{2,M}(\lambda_1(s,i/ N), \lambda_2(s, i/ N)) \\ \\ \vspace{5pt}
&  + & \,  {\xi(i)\over \lambda_2(s,i/ N)}\Big ( {\tilde {\tilde {e}}}_{1,M}(\lambda_1(s,i/ N), \lambda_2(s, i/ N)) + \\ \\ \vspace{5pt}
&& \hspace{5cm} {\tilde {\tilde {e}}}_{2,M}(\lambda_1(s,i/ N), \lambda_2(s, i/ N))\Big )
\Big ] \, d\,(\nu^N_{\rho}\times \nu^N_{\rho}) (\eta,\xi)\, ds, 
\end{eqnarray*}
and for $k=1,2$ 
\begin{eqnarray*}\vspace{3pt}
{\tilde {\tilde {\beta}}}_{k,M} (a_1,a_2) & = & \int \Big [ {\beta}_k(\eta,\xi)- {\beta}_{k,M}(\eta,\xi)\Big ]\, d (\nu^N_{a_1}\times \nu^N_{a_2})(\eta,\xi)\\ \vspace{3pt}
{\tilde {\tilde {\delta}}}_{k,M} (a_1,a_2) & = & \int \Big [ {\delta}_k(\eta,\xi)- {\delta}_{k,M}(\eta,\xi)\Big ]\, d (\nu^N_{a_1}\times \nu^N_{a_2})(\eta,\xi)\\ \vspace{3pt}
{\tilde {\tilde {e}}}_{k,M} (a_1,a_2) & = & \int \Big [ {e}_k(\eta,\xi)- {e}_{k,M}(\eta,\xi)\Big ]\, d (\nu^N_{a_1}\times \nu^N_{a_2})(\eta,\xi)\\ \vspace{3pt}
{\tilde {\tilde {r}}}_{M} (a_1,a_2) & = & \int \Big [ {r}(\eta,\xi)- {r}_{M}(\eta,\xi)\Big ]\, d (\nu^N_{a_1}\times \nu^N_{a_2})(\eta,\xi).
\end{eqnarray*}
To control the term $F(M,N,T)$ we use  lemma \ref{phi} to obtain
\begin{equation}
\label{hydro-FMNT}\lim_{M\rightarrow +\infty} \limsup_{N\rightarrow +\infty} F(M,N,T)\, =\, 0.
\end{equation}
For the rest of the paper we need the following result due to Perrut (2000). For each  
bounded function $h$ on $\N\times \N$ and  for all $x_1,x_2, y_1, y_2$ in $\R^+$, we set 
\begin{equation}\label{Gamma}
(\Gamma h)(x_1,x_2,y_1,y_2)=  \widetilde{h}(x_1,x_2)- \widetilde{h}(y_1,y_2)- {d \widetilde{h}\over dx_1}(y_1,y_2)\, (x_1-y_1) -{d \widetilde{h}\over dx_2}(y_1,y_2)\, (x_2-y_2) 
\end{equation}
\pagebreak

\begin{lemma} \ (Perrut,\ 2000)\,  \label{Anne}
Let $h(.,.)$ be a bounded function on $\N\times \N$,
 $\rho_1(.)$ and $\rho_2(.)$ be two positive bounded functions on
 $[\,0,1]$  and $J$ be a continuous function on $\R^2$. Then
 there exists $\gamma_0>0$, such that, for all $\gamma\leq \gamma_0$,
\begin{eqnarray*}
&& \displaystyle {1\over N}\sum_{i=0}^{N-1}\int J\left (\rho_1(i/ N), 
\rho_2(i/ N)\right )\, (\Gamma h)\,  \left (\eta^k(i), \xi^k(i), \rho_1(i/ N), \rho_2(i/ N)) \right  )  \times    \hfill  \\ \vspace{3pt}
&&  \hspace{4cm}  \displaystyle  f^N_t(\eta,\xi) \, d (\nu^N_{\rho} \times  \nu^N_{\rho}) (\eta,\xi) 
\leq  {1\over \gamma N} \,  \Hh\, \left [\mu^N_t| \nu^N_{\rho_1(.)}\times \nu^N_{\rho_2(.)}\right ] 
+  R_N^t(k,\gamma),
\end{eqnarray*}
with $\displaystyle \limsup_{k \to \infty}\limsup_{N\to \infty} R_N^t(k,\gamma)\leq 0.$
\end{lemma}

\begin{lemma}\label{intro-gamma} For $k=1,2$ we have
\begin{eqnarray*}  
a) & \displaystyle {1\over \lambda_k}\, \tilde{\varphi}_{k,M}(x_1,x_2) - \tilde{\beta}_{k,M}(x_1,x_2)  - \displaystyle \Big ( {x_k\over \lambda_k}-1\Big ) \tilde{\beta}_{k,M}(\lambda_1,\lambda_2)& = & 
 \hfill \cr
& & \hspace{-5cm}\displaystyle {1\over \lambda_k}\, \Gamma \varphi_{k,M}(x_1, x_2,\lambda_1,\lambda_2) - \Gamma \beta_{k,M}(x_1,x_2,\lambda_1,\lambda_2)\hfill  \crcr  \cr
b) & \lambda_k\, \tilde{\phi}_{k,M}(x_1,x_2) -  \tilde{\delta}_{1,M}(x_1,x_2) + \Big ( 
\displaystyle {x_k\over \lambda_k} - 1\Big ) \tilde{\delta}_{k,M}(\lambda_1,\lambda_2) & = &  \hfill \cr
& & \hspace{-5cm} \displaystyle {1\over \lambda_k}\, \Gamma \phi_{k,M}(x_1, x_2,\lambda_1,\lambda_2) - \Gamma \beta_{k,M}(x_1,x_2,\lambda_1,\lambda_2) \hfill  \crcr \cr
c) & \displaystyle {\lambda_1\over \lambda_2}\, \tilde{E}_{k,M}(x_1,x_2) - \tilde {e}_{k,M}(x_1,x_2)-
\Big  (\displaystyle {x_1\over \lambda_1} - {x_2\over \lambda_2} \Big ) \tilde {e}_{k,M}(\lambda_1,\lambda_2) & = & \hfill  \cr 
& & \hspace{-5cm} \displaystyle {\lambda_1\over \lambda_2}\, \Gamma E_{k,M}(x_1, x_2,\lambda_1,\lambda_2) - \Gamma e_{k,M}(x_1,x_2,\lambda_1,\lambda_2)\hfill \crcr \cr 
 d) & \displaystyle {\lambda_2\over \lambda_1}\, \tilde{R}_{M}(x_1,x_2) -  
\tilde{r}_{M}(x_1,x_2) + \displaystyle \Big ( 
{x_1\over \lambda_1} - {x_2\over \lambda_2}\Big ) \tilde{r}_{M}(\lambda_1,\lambda_2) & = & 
\hfill \crcr \cr
& & \hspace{-4.5cm} \displaystyle {\lambda_2\over \lambda_1}\, \Gamma R_{M}(x_1, x_2,\lambda_1,\lambda_2) - \Gamma r_{M}(x_1,x_2,\lambda_1,\lambda_2)
\end{eqnarray*} 
\end{lemma}
\begin{proof} We use substitution rule \eqref{cv} and the formula \eqref{Gamma} 
of $\Gamma$.  
We need only to remark that for $k=1,2$ 
$$\tilde{\varphi}_{k,M}(\lambda_1, \lambda_2)=\lambda_k \, \tilde{\beta}_{k,M}(\lambda_1,\lambda_2),\quad 
\tilde{\phi}_{k,M}(\lambda_1, \lambda_2)={1\over \lambda_k}\, \tilde{\delta}_{k,M}(\lambda_1,\lambda_2).$$
$$\tilde{E}_{k,M}(\lambda_1, \lambda_2) \, = \, {\lambda_2\over \lambda_1}\, \tilde{e}_{k,M}(\lambda_1,\lambda_2), \quad 
\tilde{R}_{M}(\lambda_1, \lambda_2) \, = \, {\lambda_1\over \lambda_2}\, \tilde{r}_{M}(\lambda_1,\lambda_2).$$
\end{proof}

\indent All  terms of the  upper bound \eqref{hydro-uperbound} of the relative entropy are evaluated in the same way,  it is enough to compute for example the first one. 
We shall replace the local functions $\varphi_M,\beta_M,\eta(.)$ and $\xi(.)$ by functions of the empirical density  of the particles in boxes of size $2k+1$, with $k$ going to infinity after $N$. This is possible thanks to the one block estimate, that is  proposition \ref{oneblock}.    

\begin{eqnarray*}
 \displaystyle  T_1 &  = &  \displaystyle  {1\over N} \displaystyle\sum_{i=0}^{N-1}\int_0^t \int 
 \Big [ {\varphi_{1,M}(\eta(i),\xi(i))\over \lambda_1(s,i/ N) } 
- \beta_{1,M}(\eta(i),\xi(i)) - \\
& & \hspace{2.25cm}  \left ({\eta(i)\over \lambda_1(s,i/ N)}-1 \right )\,  {\widetilde \beta}_{1,M}(\lambda_1(s,i/ N),\lambda_2(s,i/ N))\Big ] \,
 d\mu^N_s (\eta,\xi) \, ds\\
\hfill  & \leq &  \displaystyle {1\over N} \displaystyle\sum_{i=0}^{N-1}\int_0^t \int\Big [
{ \widetilde{\varphi}_{1,M}(\eta^k(i),\xi^k(i))\over \lambda_1(s,i/ N) } - 
{\widetilde \beta}_{1,M}(\eta^k(i),\xi^k(i)) - \\
 & & \hspace{2.25cm}  \noindent \hfill \displaystyle \left ( {\eta^k(i)\over \lambda(s,i/ N)}-1
\right ){\widetilde \beta}_{1,M}(\lambda_1(s,i/N),\lambda_2(s,i/ N)) 
 + r_N^s(M,k)\Big ] \, d\mu^N_s (\eta,\xi) \, ds,\\
\end{eqnarray*}
where $\displaystyle \limsup_{k \to \infty}\limsup_{N\to \infty} r_N^t(M,k)\leq 0.$
By a) of lemma \ref{intro-gamma},  we have
\begin{eqnarray*}
T_1 & \leq &\displaystyle  {1\over N} \displaystyle\sum_{i=0}^{N-1}\int_0^t \int \Big [
{1\over \lambda_1(s,i/ N)} (\Gamma \varphi_{1,M})  \left (\eta^k(i),\xi^k(i),\lambda_1(s,i/ N),\lambda_2(s,i/ N)\right )- \\ 
& & \hspace{4cm} (\Gamma \beta_{1,M}) \left (\eta^k(i),\xi^k(i),\lambda_1(s,i/ N),\lambda_2(s,i/ N)\right )\Big ]\, d\mu^N_s(\eta,\xi)\, ds.
\end{eqnarray*}
By lemma \ref{Anne},  there exists $ \gamma_0> 0,$ such that for all $\gamma\leq \gamma_0$
$$ T_1 \leq {2\over \gamma N}\int_0^t  \, \Hh \, \left [\mu^N_s| \nu^N_{\lambda_1(s,.)}\times \nu^N_{\lambda_2(s,.) }\right ] \, ds + 
 R_N^t(k,\gamma) + r_N^t(M,k),$$
where $\displaystyle \limsup_{k \to \infty} \ \limsup_{N\to \infty} R_N^t(k,\gamma)\leq 0$. Then
\begin{equation}\label{entro4} 
\begin{array}{rcl}\vspace{5pt}
& & \displaystyle {1\over N}\, \Hh\, \left [\mu_t^N| \nu^N_{\lambda_1(t,.)}\times \nu^N_{\lambda_2(t,.)}\right ]\leq  \displaystyle {1\over N }\, \Hh\, \left  [\mu^N| \nu^N_{m_1(.)}\times \nu^N_{m_2(.)}\right ] +
r_N^t(M,k) + R_N^t(k,\gamma) +  \\ \vspace{5pt}
& & \hspace{3cm}  F(M,N,T) + \displaystyle {14\over \gamma N} \int_0^t \, \Hh\,\left [\mu^N_t| \nu^N_{\lambda_1(t,.)}\times \nu^N_{\lambda_2(t,.)}\right ]ds + o(1/ N)
\end{array}
\end{equation}
Finally,  hypothesis \eqref{ci} and  Gronwall lemma imply  \eqref{zero1}, which ends the proof.
\end{proof}


\subsection{Proof of Theorem \ref{hydro-hthiv} (extension to infinite volume) }
\label{hydro-Sec: volumeinfini}
\indent To extend Theorem \ref{hydro-hth} to infinite volume, that is to all space $\Z$, 
we follow the same strategy as in Perrut (2000), and in Landim and Yau, (1995), we make a coupling  
between two processes: the first one  $(\eta^1_t, \xi^1_t)_{t\geq 0}$ on $\Z$ with $\mu^N$ as initial distribution and 
the second one $(\eta^2_t, \xi^2_t)_{t\geq 0}$ on  $\T_{C N}=\{-C N, \ldots, C N \}$ with 
$\mu^N$ restricted to $\T_{C N}$ as initial distribution.
We will prove that when $N$ goes to infinity and $C$ is large the ``difference" between those two 
processes is small in a sense to be specified later. 
 
\indent To couple $(\eta^1_t, \xi^1_t)_{t\geq 0}$ and $(\eta^2_t, \xi^2_t)_{t\geq 0}$ we distinguish between two types of particles: the coupled ones and the non-coupled ones.
More precisely, at site $x$, the $\eta_t^1(x)$-particles are divided into $\eta^*_t(x)$ and $\eta^{1*}_t(x)$. The  $\eta^*_t(x)$-particles  are associated to particles of $\eta_t^2(x)$, these couples of particles move together. All the other particles stay single. Initially $\eta_0^*(x)=\eta_0^1(x)\wedge \eta_0^2(x)$ for all $x\in \T_{C N}$. We set $\eta_0^1(x)= \eta_0^*(x)+ \eta_0^{1*}(x)$,
$\eta_0^2(x)= \eta_0^*(x)+ \eta_0^{2*}(x)$ and do the same for $\xi_0^1$ and $\xi_0^2$.  

\indent The diffusion part of the coupled generator $\overline{\Omega}_N$, is denoted by  $\overline {\Omega}_N\super {\cal D}$, where
$$\overline{\Omega}_N\super {\cal D}  =\overline{\Omega}_N\super {{\cal D},1} +\overline{\Omega}_N\super {{\cal D},2},$$ 
and  $\overline{\Omega}_N\super {{\cal D},1}$   describes  at sites $|x| < CN$ the evolution by:

$$ \begin{array}{rcl} \vspace{2pt}
\overline{\Omega}_N\super {{\cal D},1} f(\eta^*,\eta^{*1}, \eta^{*2}) & =  &  \displaystyle  \sum_{ \substack {|x|< C N\\ y \in \T_{CN} } } p(x,y)\,   \eta^*(x)\, \Big[
f((\eta^*)^{x,y},\eta^{*1}, \eta^{*2})- f(\eta^*,\eta^{*1}, \eta^{*2})\Big ]  \\ \vspace{2pt}
& + &  \displaystyle  \sum_{ \substack {|x|< C N\\  y \in \T_{CN}}}  p(x,y)\,  \eta^{1*}(x) \wedge \eta^{*2}(x) \, \Big[ f(\eta^*,(\eta^{*1})^{x,y} , (\eta^{*2})^{x,y})- f(\eta^*,\eta^{*1}, \eta^{*2})\Big ]  \\ \vspace{2pt}
& + &  \displaystyle  \sum_{\substack {|x|< C N \\ y \in \T_{CN}}} p(x,y)\,  \Big (\eta^{1*}(x)- \eta^{*2}(x)\Big )^+ \, 
\Big[ f(\eta^*,(\eta^{*1})^{x,y} , \eta^{*2})- f(\eta^*,\eta^{*1}, \eta^{*2})\Big ]  
\\ \vspace{2pt}
& + &  \displaystyle  \sum_{\substack {|x| < C N \\ y\in \T_{CN}}}  p(x,y)\,  \Big (\eta^{*2}(x)- \eta^{*1}(x)\Big )^+ \, \Big[ f(\eta^*,  \eta^{*1}, (\eta^{*2})^{x,y})- f(\eta^*,\eta^{*1}, \eta^{*2})\Big ].\end{array} $$
At  site $x = C N$, the particles of the two processes jump outside $\{-C N,\ldots, C N\}$ independently. Those of $\eta^1$ arrive at ${C N +1}$ and the others at $-C N$.   The coupled generator   $\overline{\Omega}_N\super{{\cal D},2}$ of infected individuals $\xi$ evolves according to the same rules. 

\indent The reaction part of the coupled generator $\overline{\Omega}_N$ is denoted by  
$\overline {\Omega}_N\super {\cal R}$ and defined for all cylinder function as the sum of 
$\overline{\Omega}_N\super {{\cal R}, i}$, $i=1,\ldots,5$. Let  
$\overline{ \Omega}_N\super {{\cal R}, 1}$  be the coupled generator of birth and death of healthy individuals described at sites $|x|<CN$ by:  

At  rate $\beta_1(\eta^1(x),\xi^1(x))\wedge \beta_1(\eta^2(x),\xi^2(x))$ (respectively, $\delta_1(\eta^1(x),\xi^1(x))\wedge\delta_1(\eta^2(x),\xi^2(x))$) two coupled particles 
are created (respectively, removed),  at rate  
$\big (\beta_1(\eta^1(x),\xi^1(x)) -\beta_1(\eta^2(x),\xi^2(x))\big)^+$ 
(respectively, $\big (\delta_1(\eta^2(x),\xi^2(x))- \delta_1(\eta^1(x),\xi^1(x)) \big )^+$)
a particle of $\eta^{1*}$  is created (respectively, removed), and at rate 
$\big (\beta_1(\eta^2(x),\xi^2(x))- \beta_1(\eta^1(x),\xi^1(x)) \big )^+$ 
(respectively,  $\big (\delta_1(\eta^2(x),\xi^2(x))- \delta_1(\eta^1(x),\xi^1(x)) \big )^+$)
a particle of $\eta^{2*}$  is created (respectively, removed). 
In a symmetric way  we define the coupled generator describing birth and death of infected individuals.
In the same way, we define the coupled process  of recoveries and infection (inside infection, outside infection and   recoveries of infected individuals).

\indent We denote by $\overline{E}_{\mu^N}$ the expectation of the coupled process 
$\overline{\Omega}_N$
starting from $\mu^N$. For notational simplicity, we assume that $\alpha_1 + \alpha_2\leq 1$, and we set $$\zeta_s^*(x) = \eta^{*1}_s(x) + \eta^{*2}_s(x)+ \xi^{*1}_s(x) +\xi^{*2}_s(x).$$ 
For $x \in \T_{C \,N} $,  since $\zeta(x)$ is constant for the coupled generators for the outside and inside infections  and the  recoveries,  we have:

$$\overline{\Omega}_N\super{\cal R} (\zeta^*(x) )= \overline{\Omega}_N\super{{\cal R},1} (\zeta^*(x) )
\leq   (\alpha_1+ \alpha_2)\, |\eta^1(x)-\eta^2(x)| + (\alpha_1+ \alpha_2)\, 
|\xi^1(x)-\xi^2(x)|$$
 Thus, 
\begin{equation} \label{expectation}
\overline{\Omega}_N\super {\cal R}( \zeta_s^*(x) ) \,  \leq  \,  \zeta_s^*(x).
\end{equation}
Furthermore, since the death rates are larger than the birth rates 
it exists a real $c_0>0$ such that
$$\overline{\Omega}_N\super {\cal R}\big (\eta^{1}(x) + \xi^{1}(x) \big) \leq 
\beta_1(\eta^1(x),\xi^1(x)) - \delta_1(\eta^1(x),\xi^1(x)) 
+ \beta_2(\eta^1(x),\xi^1(x)) - \delta_2(\eta^1(x),\xi^1(x))\leq c_0,$$
and since $\displaystyle {d\over dt} \overline{E}_{\mu^N}\Big [f(\eta_t,\xi_t)\Big ] = \overline{E}_{\mu^N}\Big [\overline{\Omega}_N\super {\cal R}f(\eta_t,\xi_t)\Big]$, 
\begin{equation}\label{Maj1}
\overline{E}_{\mu^N} \Big [\overline{\Omega}_N\super{\cal R}  \big ( \eta^1_t(x)+ \xi^1_t(x) \big )\Big ] \leq c_0, 
\quad  \overline{E}_{\mu^N} \big ( \eta^1_t(x)+ \xi^1_t(x) \big ) \leq \, M + t\, c_0,
\end{equation}
and 
\begin{equation}\label{Maj3}
\overline{E}_{\mu^N}(\zeta_s^*(x) ) \leq 2\, M \, + \,  2\,t\, c_0 := K_1.
\end{equation}
Let $A\in\N$ be fixed. Now we have all the necessary  tools to bound above the discrepancy between the two processes in the box $\Lambda_{AN}=\{-NA,\ldots, NA\}$. By following the same steps as in Perrut (1999) we prove first 
by using \eqref{expectation}, \eqref{Maj1} and \eqref{Maj3} that
\begin{equation}\label{hydro-diff22}\lim_{C\to\infty} \lim_{N\to\infty} \overline{E}_{\mu^N}\Big [ {1\over N} \sum_{x\in \Lambda_{AN}} \zeta^*_t(x)\Big ] = 0,
\end{equation}
and then theorem \ref{hydro-hthiv}.  

\vspace{2mm}

\noindent\textbf{Acknowledgment}.
We thank Ellen Saada for many useful advice and fruitful discussions.


\end{document}